\documentclass[11pt,reqno,a4paper]{amsart}

\usepackage{fouriernc}
\usepackage[T1]{fontenc}

\usepackage{enumerate}
\usepackage{amssymb}
\usepackage{amsmath}
\usepackage{amscd}
\usepackage{amsthm}
\usepackage{amsfonts}
\usepackage{graphicx}
\usepackage[all]{xy}
\usepackage{verbatim}
\usepackage{hyperref}
\usepackage{amssymb}

\DeclareMathOperator{\AP}{AP}
\DeclareMathOperator{\B}{B}
\DeclareMathOperator{\E}{E}

\DeclareMathOperator{\Ext}{Ext}
\DeclareMathOperator{\Erg}{Erg}
\DeclareMathOperator{\Isom}{Isom}
\DeclareMathOperator{\supp}{supp}
\DeclareMathOperator{\Syn}{Syn}
\DeclareMathOperator{\PSyn}{PW-Syn}
\DeclareMathOperator{\Bohr}{Bohr}
\DeclareMathOperator{\PBohr}{PW-Bohr}

\newtheorem{theorem}{Theorem}[section]

\newtheorem{lemma}[theorem]{Lemma}
\newtheorem{corollary}[theorem]{Corollary}
\newtheorem{proposition}[theorem]{Proposition}

\theoremstyle{definition}

\theoremstyle{definition}

\theoremstyle{definition}

\theoremstyle{definition}
\theoremstyle{definition}\newtheorem{remark}[theorem]{Remark}
\theoremstyle{definition}

\newcommand{\eps}{\epsilon}

\newcommand{\cB}{\mathcal{B}}
\newcommand{\cC}{\mathcal{C}}

\newcommand{\cF}{\mathcal{F}}

\newcommand{\cH}{\mathcal{H}}

\newcommand{\cL}{\mathcal{L}}
\newcommand{\cM}{\mathcal{M}}

\newcommand{\cP}{\mathcal{P}}

\newcommand{\cS}{\mathcal{S}}

\newcommand{\bG}{\mathbb{G}}

\newcommand{\bR}{\mathbb{R}}
\newcommand{\bZ}{\mathbb{Z}}

\newcommand{\bF}{\mathbb{F}}

\newcommand{\bT}{\mathbb{T}}

\newcommand{\goL}{\mathfrak{L}}
\newcommand{\goR}{\mathfrak{R}}
\newcommand{\goS}{\mathfrak{S}}

\newcommand{\oA}{\overline{A}}
\newcommand{\oB}{\overline{B}}
\newcommand{\oC}{\overline{C}}

\newcommand{\oT}{\overline{T}}
\newcommand{\oU}{\overline{U}}

\newcommand{\olambda}{\overline{\lambda}}
\newcommand{\oeta}{\overline{\eta}}

\newcommand{\ra}{\rightarrow}

\newcommand{\onto}{\xymatrix{\ar@{>>}[r]&}}
\newcommand{\da}[4]{\xymatrix{#1 \ar@<.5ex>[r]^{#2} \ar@<-.5ex>[r]_{#3} & #4}}

\newcommand{\qand}{\quad \textrm{and} \quad}
\newcommand{\qqand}{\qquad \textrm{and} \qquad}

\newcommand{\cFS}{\mathcal{F}\mathcal{S}}
\newcommand{\cBo}{\emph{Bohr}}

\numberwithin{equation}{section}

\begin{document}

\title{Product set phenomena for countable groups}


\author{Michael Bj\"orklund}
\address{Department of Mathematics , ETH Z\"urich, Z\"urich, Switzerland}
\curraddr{Department of Mathematics, Chalmers, Gothenburg, Sweden}
\email{micbjo@chalmers.se}
\thanks{}

\author{Alexander Fish}
\address{School of Mathematics and Statistics, University of Sydney, Australia}
\curraddr{}
\email{alexander.fish@sydney.edu.au}
\thanks{}



\date{}

\dedicatory{}

\begin{abstract}
We develop in this paper novel techniques to analyze  
local combinatorial structures in product sets of two subsets of a 
countable group which are "large" with respect to certain 
classes of (not necessarily invariant) means on the group. 
Our methods heavily utilize the theory of C*-algebras and random walks on groups. As applications of our methods, we extend and quantify 
a series of recent results by Jin, Bergelson-Furstenberg-Weiss,
Beiglb\"ock-Bergelson-Fish, Griesmer and DiNasso-Lupini 
to general countable groups. 
\end{abstract}

\maketitle


\section{Introduction}

\subsection{General comments}
Let $G$ be a group and let $\goL, \goR$ and $\goS$ be given sets 
of subsets of $G$. We shall think of $\goL$ and $\goR$ as defining 
two (possibly different) classes of \emph{large} subsets of the group and 
the elements of $\goS$ will be regarded as the \emph{structured} 
subsets of $G$. 

In this paper, the term \emph{product set phenomenon} 
(with respect to the sets $\goL, \goR$ and $\goS$) will refer 
to the event that whenever $A \in \goL$ and $B \in \goR$, 
then their \emph{product set} $AB$, defined by
\[
AB := \Big\{ a \cdot b \, : \, a \in A, \, \, b \in B \Big\}
\]
belongs to $\cS$. If this happens, we shall say that the pair 
$(\goL,\goR)$ is \emph{$\goS$-regular}.  

Perhaps the first occurrence of a (non-trivial) product set phenomenon recorded in the literature is the following classical observation, which is often attributed to Steinhaus (see e.g. \cite{St}): Let $G$ be a locally compact group with left Haar measure $m$ and define
\[
\goL := \Big\{ A \in \cB(G) \, : \, m(A) > 0 \Big\},
\]
where $\cB(G)$ denotes the set of Borel sets of $G$. Let $\goS$ denote
the set of all subsets of $G$ with non-empty interior. Then the pair 
$(\goL,\goL)$ is $\goS$-regular, that is to say, the product set of any 
two Borel sets with positive Haar measures contains a non-empty 
open set. 

\subsection{Structured sets in countable groups}
This paper is concerned with product set phenomena in \emph{countable}
groups. To explain our main results, we first need to define what classes 
of \emph{large} sets and \emph{structured} sets we shall consider. We begin by describing our choices of the structured sets. 

Let $G$ be a countable group. A set $T \subset G$ is \emph{right thick}
if whenever $F \subset G$ is a finite subset, then there exists $g$ in $G$
such that  
\[
F \cdot g \subset T .
\]
We say that a set $C \subset G$ is \emph{left syndetic} if it has non-trivial intersection with any right thick set, or equivalently, if there exists 
a finite set $F \subset G$ such that $FC = G$. Let $\Syn$ denote the set
of all left syndetic subsets of $G$ and define the class of 
\emph{right piecewise left syndetic sets} $\PSyn$ by
\[
\PSyn := 
\Big\{
C \cap T \, : \, \textrm{$C$ is left syndetic and $T$ is right thick}.
\Big\}.
\]
Equivalently, a set $C \subset G$ is right piecewise left syndetic if there 
exists a finite set $F \subset G$ such that the product set $FC$ is right thick. 

A particularly nice sub-class of left syndetic sets is formed by the so called \emph{Bohr sets}. 
Recall that a set $C \subset G$ is \emph{Bohr} if there exist a compact 
Hausdorff group $K$, an epimorphism $\tau : G \ra K$ (a homomorphism
with dense image) and a non-empty open set $U \subset K$ such that 
\[
C \supset \tau^{-1}(U).
\]
Let $\Bohr$ denote the set of all Bohr sets in $G$ and define the class 
$\PBohr$ of \emph{right piecewise Bohr sets}  by
\[
\PBohr := 
\Big\{
C \cap T \, : \, \textrm{$C$ is Bohr and $T$ is right thick}.
\Big\}.
\]
The classes $\Syn$, $\PSyn$, $\Bohr$ and $\PBohr$ will all be used as structured subsets of $G$ in this paper. 

\subsection{Large sets in amenable groups}
The \emph{large} subsets of $G$ will primarily be defined in terms of means on $G$. Recall that a \emph{mean} on a countable $G$ is a positive linear functional $\lambda$ on $\ell^{\infty}(G)$ with norm one. Alternatively, we can  think of a mean as a \emph{finitely additive} probability measure $\lambda'$ on $G$ via the correspondence 
\[
\lambda'(C) = \lambda(\chi_C), \quad C \subset G,
\]
where $\chi_C$ denotes the indicator function of the set $C$. Let 
$\cM(G)$ denote the set of all means on $G$. If $\cC \subset \cM(G)$,
then we define the upper and lower \emph{$\cC$-density} of  
a set $C \subset G$ by
\[
d^{*}_{\cC}(C) = \sup_{\lambda \in \cC} \lambda'(C)
\qand
d_{*}^{\cC}(C) = \inf_{\lambda \in \cC} \lambda'(C)
\]
respectively. We say that a set is \emph{$\cC$-large} if $d^{*}_{\cC}(C)$ is positive and \emph{$\cC$-conull} if $d_{\cC}^{*}(C)$ equals one. \\

We note that $G$ acts on the set $\cM(G)$ via the left regular representation $\rho$ by 
\[
(\rho(g)^*\lambda)(\phi) = \lambda(\rho(g)\phi), \quad \textrm{for $g\in G$ and $\phi \in \ell^\infty(G)$}.
\] 
Recall that a countable group $G$ is \emph{amenable} if there exists a left invariant mean on $G$, i.e. if the set 
\[
\cL_{G} 
:= 
\Big\{ \lambda \in \cM(G) \, : \, \rho(g)^{*} \lambda = \lambda, \quad \forall \, g \in G \Big\},
\]
where $\rho$ denotes the left regular representation on $\ell^{\infty}(G)$
is non-empty. The class of amenable groups contains all abelian groups, 
all locally finite groups and all groups of subexponential growth. The free
groups on at least two generators are well-known examples of non-amenable groups. We stress that the upper and lower $\cL_G$-densities on an amenable
group are often referred to as the \emph{upper and lower Banach densities} in the literature. 

In the very influential classical paper \cite{Fo}, F\o lner (inspired by earlier works of Bogolyubov in \cite{Bo}) observed that if $G$ is a countable amenable group and $A \subset G$ is $\cL_G$-large, then its \emph{difference set} $AA^{-1}$ is left syndetic and right piecewise 
Bohr. In fact, F\o lner showed that one can always find a Bohr set $B$, which contains the identity element, and a
right thick set $T$ with the property that $T \cap S$ is right thick for 
\emph{every} right thick set $S$ in $G$ such that
\[
AA^{-1} \supset B \cap T. 
\]
More recently, Jin in \cite{Jin}, Bergelson, Furstenberg and Weiss in \cite{BFW} (for the integers) and Beiglb\"ock, Bergelson  and the second author of this paper in \cite{BBF}, showed that whenever $A$ and $B$
are $\cL_G$-large subsets of a countable amenable group $G$, then 
the product set $AB$ is  right piecewise Bohr (and thus right piecewise left
syndetic). In other words, the pair $(\cL_G,\cL_G)$ is 
$\PBohr$-regular (and thus $\PSyn$-regular). Note however that already easy examples show that $(\cL_G,\cL_G)$ is not $\Syn$-regular.

\subsection{Pit-falls for non-amenable groups}
Let $G$ be a countable group and suppose $\mu$ is a probability 
measure on $G$ such that $\mu(g) = \mu(g^{-1})$ for all $g$ and
the support of $\mu$ generates $G$. We shall call 
such a measure \emph{adapted} and refer
to the pair $(G,\mu)$ as a \emph{measured group}. If $\rho$ denotes the left regular representation of $G$ on $\ell^{\infty}(G)$, then we let $\rho(\mu)$ denote the operator
\[
\rho(\mu)\varphi(x) = \int_{G} \varphi(g^{-1}x) \, d\mu(g), \quad 
\varphi \in \ell^{\infty}(G).
\] 
Classical arguments (e.g. Kakutani's Fixed Point Theorem) show that 
the set 
\[
\cL_{\mu} 
= 
\Big\{ \lambda \in \cM(G) \, : \, \rho(\mu)^{*} \lambda = \lambda \Big\}
\]
is non-empty for every measured group $(G,\mu)$. We think of $\cL_\mu$
as a substitute for $\cL_G$ when the group is not amenable. Indeed, if $G$
is amenable, then it is well-known (see e.g. \cite{KV} or \cite{Ro}) that there
exists a probability measure $\mu$ as above, such that $\cL_\mu = \cL_G$. The study of $\cL_\mu$-large sets was initiated by Furstenberg and Glasner in \cite{FG}.

One of the aims of this paper is to extend the results mentioned above (for amenable groups) to the setting of general countable groups and $\cL_\mu$-densities. There are however some serious pit-falls in the non-amenable case which the reader should be aware of before we formulate our results. 

Note that the F\o lner Theorem implies that whenever $G$ is a countable \emph{amenable} group and $A$ and $B$ and $C$ are $\cL_G$-large sets, then the intersection
\[
AA^{-1} \cap BB^{-1} \cap CC^{-1} \quad \textrm{is $\cL_G$-large}.
\]
Let $G$ denote the free group on three (free) generators $\{a,b,c\}$ 
and let $A$, $B$ and $C$ denote the sets of all elements in $G$, 
viewed as reduced words written from left to right, beginning with 
the letters $a$, $b$ and $c$ respectively. One readily checks that $A, B$ and $C$ are left syndetic, so in particular they are $\cL_\mu$-large for every probability measure $\mu$. However, 
\[
AA^{-1} \cap BB^{-1} \cap CC^{-1} = \{e\},
\]
which shows that the F\o lner Theorem does not (naively) extend to non-amenable groups. In fact, the situation is even worse. In Subsection 
\ref{pitfall} we shall construct a subset $A$ of the free group on two 
generators, which is $\cL_\mu$-large for every adapted probability measure
$\mu$ on $G$, such that the difference set $AA^{-1}$ is \emph{not} left syndetic. 

\subsection{Furstenberg-Poisson means}
Let $(G,\mu)$ be a countable measured group. A $\mu$-integrable function $f$ on $G$ is called \emph{left $\mu$-harmonic} if
\[
f(g) = \int_G f(hg) \, d\mu(h), \quad \forall \, g \in G. 
\]
Let $\cH^{\infty}_l(G,\mu)$ denote the space of all \emph{bounded}
left $\mu$-harmonic functions. If $G$ is non-amenable, then it is 
well-known that $\cH^{\infty}_l(G,\mu)$, equipped with the supremum
norm, is a non-separable Banach space. Define the set of (left) 
\emph{Furstenberg-Poisson means} on $G$ by
\[
\cF_\mu := 
\Big\{ 
\lambda \in \cL_\mu \, : \, \lambda |_{\cH^{\infty}_{l}(G,\mu)} = \delta_{e} \Big\}.
\]
It is easy to construct elements in $\cF_\mu$. Indeed, for $n \geq 1$, 
we define the mean
\[
\lambda_n(\varphi) := \frac{1}{n} \sum_{k=1}^{n} \int_G \varphi \, d\mu^{*k}, \quad \forall \, \varphi \in \ell^{\infty}(G),
\]
where $\mu^{*k}$ denotes the $k$-th convolution power of $\mu$. 
One readily checks that every accumulation point $\lambda$ of the sequence $(\lambda_n)$ in the weak*-topology on $\cM(G)$ belongs to $\cL_\mu$ and if $f$ is a left $\mu$-harmonic function on $G$, then 
$\lambda_n(f) = f(e)$ for all $n$, and thus $\lambda \in \cF_\mu$. In 
particular, if $B \subset G$ and
\[
\varlimsup_{n \ra \infty} \, \frac{1}{n} \sum_{k = 0}^{n-1} \mu^{*k}(B) > 0,
\]
then $B$ is $\cF_\mu$-large (and hence $\cL_\mu$-large). 

\subsection{Fourier-Stiltjes means}
Let $G$ be a countable group and let $\pi$ be a unitary representation 
of $G$ on a Hilbert space $\cH$. Given $x, y \in \cH$, the function 
\[
\varphi(g) = \big\langle y, \pi(g)x \big\rangle
\]
is called a \emph{matrix coefficient} of $\pi$. The space of all matrix
coefficients on $G$ (when $\pi$ ranges over all unitary representations 
of $G$) is easily seen to be a $*$-sub-algebra of $\ell^{\infty}(G)$, i.e. 
a linear subspace which is closed under pointwise multiplication and taking complex conjugates. We shall refer to this algebra as the \emph{Fourier-Stiltjes algebra} of $G$ and denote it by $\B(G)$.

A classical result due to Ryll-Nardzewski (see e.g. \cite{Gl}) asserts there exists a
\emph{unique} left invariant mean $\lambda_o$ on $\B(G)$. We define the
space $\cFS$ of \emph{Fourier-Stiltjes means} on $G$ by
\[
\cFS := \Big\{ \lambda \in \cM(G) \, : \, \lambda |_{\B(G)} = \lambda_o \Big\}.
\]
Clearly, if $G$ is amenable, then every left invariant mean belongs to $\cFS$. In fact, it is not hard to show that $\cL_\mu \subset \cFS$ for
every measured group $(G,\mu)$. Indeed, note that if $\varphi$ is a 
matrix coefficient of the form above and $\lambda$ belongs to $\cL_\mu$, then
\[
\lambda(\varphi) = \big\langle y, \bar{x} \big\rangle,
\]
where 
\[
\bar{x} := \int_{G} \pi(h)x \, d\lambda(h)
\]
Note that $\bar{x}$ satisfies
\begin{eqnarray*}
\int_G \pi(g) \bar{x} \, d\mu(g) 
&=& 
\int_G \Big( \int_{G} \pi(gh) x \, d\lambda(h) \Big) \, d\mu(g) \\ 
&=& 
\int_{G} \pi(h) x \, d(\rho(\mu)^{*} \lambda)(h) = \bar{x}.
\end{eqnarray*}
By the strict convexity of the unit ball in $\cH$, we conclude that $\bar{x}$ is $\pi(G)$-invariant, and thus the restriction 
of $\lambda$ to $\B(G)$ must be left $G$-invariant. By the uniqueness of 
the left invariant mean on $\B(G)$, the restriction must coincide with 
$\lambda_o$. In particular, 
\[
\cF_\mu \subset \cL_\mu \subset \cFS.
\]
For an alternative proof of the second inclusion, we refer the reader to \cite{FG}.

We stress that the set of $\cFS$-large subsets of $G$ is considerably 
larger than the set of $\cL_\mu$-large subsets for any adapted 
measure $\mu$. See Subsection \ref{connections} for 
further comments. 

\subsection{Statements of main results}

As we have seen from the discussions above, great care needs to be taken 
when trying to formulate sensible extensions of the F\o lner Theorem to non-amenable 
groups. In this paper, we shall present two results in this direction. The first one
partly extends the syndeticity claim in the F\o lner Theorem, and the second partly extends the "local" Bohr containment. 

Recall that a set $C \subset G$ is \emph{right piecewise left syndetic} if there exists a finite set $F \subset G$ such that $FC$ is right thick. 

\begin{theorem}
\label{thm1}
Let $(G,\mu)$ be a countable measured group. Suppose $A \subset G$ is $\cF_\mu$-large
and $B \subset G$ is $\cL_\mu$-large. Then $AB$ is right piecewise left syndetic.
\end{theorem}

A more symmetrical statement, which does not involve $\cF_\mu$-densities, can also be made. The deduction of this statement from 
the theorem above will be briefly explained in Subsection 
\ref{ss ergodic theorem}. However, we stress that there are more direct routes to proving this result. 

\begin{corollary}
Let $(G,\mu)$ be a countable measured group. If $A \subset G$ is 
$\cL_\mu$-large, then the difference set $AA^{-1}$ is right piecewise left syndetic, but not necessarily left syndetic. 
\end{corollary}

Our second theorem requires a bit more set up. Let $G$ be a countable group and let 
$B \subset G$. The \emph{Bebutov algebra of $B$}, here denoted by $\cB_B$, is defined
as the smallest left $G$-invariant sub-C*-algebra of $\ell^{\infty}(G)$ containing the constant functions and the indicator function $\chi_B$
of the set $B$.

We shall say that $B$ is \emph{strongly non-paradoxical} if
there exists a left $G$-invariant mean $\eta$ on $\cB_B$ such that $\eta'(B)$ is positive. 
Let $\cL_B$ denote the set of all means on $G$ whose restrictions to $\cB_B$ are left 
$G$-invariant. The \emph{non-paradoxical} density of $B$ is now defined as
\[
d^*_{np}(B) = \sup_{\lambda \in \cL_B} \lambda'(B).
\]
This definition may seem rather far-fetched, so we shall take a moment here to motivate why
we have introduced it. First note that if $G$ is amenable, then $\cB_B$ always admits a 
left $G$-invariant mean and thus, for the left upper Banach density $d^*_{\cL_G}$ on $G$, we have the identity
\[
d^*_{\cL_G}(B) = d^*_{np}(B),
\]
so the reader who wishes to stay within the class of amenable groups (for which the result
which we shall state below is also new) can simply think of the non-paradoxical density as 
a roundabout re-packaging of the standard upper Banach density. \\

In order to better understand which kind of sets we wish to exclude by restricting our attention
to strongly non-paradoxical sets (and to partly motivate the name), we can consider the group
$\bF_2$, the free group on two (free) generators, which we shall denote by $a$ and $b$. Let 
$A$ be the set of all group elements whose reduced form (from left to right) begins with either 
$a$ or $a^{-1}$. Note that 
\[
A \cup a \cdot A \cup a^{-1} \cdot A = \bF_{2},
\]
so in particular, $A$ is left syndetic. Moreover, 
\[
\forall \, m \neq n, \quad b^{m} \cdot A \cap b^{n} \cdot A \neq \emptyset, 
\]
so all the left translates of the form $b^{n} \cdot A$ are \emph{pairwise disjoint}. If there exists a mean 
$\lambda$ on $G$ such that $\lambda'(g \cdot A) = \lambda'(A)$ for all $g \in G$ (that is 
to say, the set $\cL_A$ is non-empty), then 
\[
\lambda'\Big( \bigcup_{k=1}^n b^{k} \cdot A\Big) = n \cdot \lambda'(A) \leq 1,
\]
for all $n$, which clearly forces $\lambda'(A) = 0$. However, we also have
\[
1 = \lambda'(A \cup a \cdot A \cup a^{-1} \cdot A) \leq 3 \cdot \lambda'(A),
\]
and thus $\lambda'(A) \geq \frac{1}{3}$. This contradiction clearly implies that $A$ is 
\emph{not} strongly non-paradoxical. In essence, the notion of strong non-paradoxicality is designed to exclude exactly this kind of sets. \\

On the other hand, every countable group $G$ admits a wealth of 
strongly non-paradoxical sets. One way to see this is to consider
the space $2^G$ of all subsets of $G$, equipped with the (compact) product topology, and the clopen subset
\[
U := \Big\{ A \subset G \, : \, e \in A \Big\} \subset 2^G.
\]
Note that $G$ acts on $2^G$ by translations on the right, and the space $\cP_G(2^G)$ of $G$-invariant Borel probability 
measures is always non-empty. Furthermore, the weak*-open
set
\[
\tilde{U} := \Big\{ \nu \in \cP_G(2^G) \, : \, \nu(U) > 0 \Big\}
\]
always contains uncountably many ergodic measures, and it is not hard to show that for every $\nu$ in $\tilde{U}$, there 
exists a $\nu$-conull Borel set $X_\nu \subset 2^G$ with the 
property that \emph{every} $B \in X_\nu$ is strongly non-paradoxical. In fact, the converse holds as well: If $A \subset G$ is strongly non-paradoxical, then there exist an ergodic $G$-invariant probability measure $\nu$ on $2^G$ and a conull Borel set $X_\nu \subset 2^G$, which contains $A$. \\

For more \emph{explicit} examples of strongly non-paradoxical sets in genuinely \emph{non-amenable} situations, we can consider the free group $G$ on two (or more) generators. 
Note that every compact semisimple Lie group $K$ contains a dense subgroup isomorphic to $G$ (see e.g. \cite{BrGe}). For an explicit construction of dense free subgroups in orthogonal 
groups, we refer the reader to \cite{Sa}. Now, if $U$ is any non-empty open subset of $K$, then it is not hard to show that 
$A := G \cap U$ is a strongly non-paradoxical subset of $G$. \\

The last examples are special cases of a more general construction. Recall that the \emph{Bohr compactification} $(bG,\iota_o)$ of a countable group $G$ consists of a (possibly trivial) 
compact Hausdorff group $bG$ and a homomorphism $\iota_o : G \ra bG$ with dense image such that whenever $K$ is a compact Hausdorff group, together with a homomorphism $\iota : G \ra K$ with dense image, then there exists a surjective homomorphism $\iota_K : bG \ra K$ such that $\iota_K \circ \iota_o = \iota$. Note that a set $B \subset G$ is a Bohr set according to the definition above if and only if there exists a non-empty open set $U \subset bG$ such that $B = \iota_o^{-1}(U)$. Clearly, every Bohr set is strongly non-paradoxical. \\

We need a final definition before we can state our second theorem. If $K$ is a compact 
Hausdorff group and $U$ is a non-empty open set, then there is a finite set $F$ such that
$FU = K$. The minimal cardinality will be referred to as the
\emph{syndeticity index of $U$} in $K$ and denoted by $s_K(U)$. Formally, 
\[
s_K(U) := \min\Big\{ |F| \, : \, F \subset K \qand FU = K \Big\}.
\]
We shall think of $s_K(U)$ as a measurement of the "regularity" of the open set. Intuitively, 
if $U$ is a random union of small disjoint open sets (a "porous" set), then we expect the 
syndeticity index of this set to be quite large relative to the Haar measure of the set, while more 
"connected" subsets (which look more like intervals or open subgroups) are likely to have 
smaller syndeticity indices. \\

We can now state our last theorem which is new already in the case of the integers (where
strong non-paradoxicality translates to $\cL_{\bZ}$-large and $d^*_{np}$ equals the upper
Banach density). 

\begin{theorem}
\label{thm2}
Let $G$ be a countable group. Suppose $A \subset G$ is $\cFS$-large and $B \subset G$ is strongly left non-paradoxical. Then there exist an open set $U \subset bG$ with
\[
s_{bG}(U) \leq \Big\lfloor \frac{1}{d^*_{\cFS}(A) \cdot d^*_{\textrm{np}}(B)} \Big\rfloor
\]
and a right thick set $T \subset G$ such that
\[
AB \supset \iota_o^{-1}(U) \cap T.
\]
In particular, if $G$ is amenable and $A \subset G$ is $\cFS$-large, then 
$AB$ is a right piecewise Bohr set for every $\cL_G$-large set $B \subset G$.
\end{theorem}

\subsection{Connection to earlier results}
\label{connections}

We have already mentioned in the introduction the classical papers of Bogliouboff and 
of F\o lner which have motivated a plethora of later works on product sets in groups. In 
this short section, we wish to draw the reader's attention to some more recent works 
which our main results generalize. \\

Theorem \ref{thm1} is inspired by the following result of R. Jin from 2002 which was 
established by a clever non-standard analytic reduction to the setting of the F\o lner Theorem.

\begin{theorem}
\cite{Jin}
Suppose $A, B \subset \bZ$ are $\cL_{\bZ}$-large subsets. Then $AB$ is piecewise syndetic. 
\end{theorem}

More recently, Di Nasso and Lupini \cite{diN} gave the following quantification of Jin's result - for any (not necessarily countable) amenable group - which motivated Theorem \ref{thm2} of this paper. 

\begin{theorem}[\cite{diN}]
Let $G$ be a countable amenable group and suppose $A, B \subset G$ are $\cL_{G}$-large subsets. Then there exists a finite set $F \subset G$ with
\[
|F| \leq  \Big\lfloor \frac{1}{d^*_{\cL_G}(A) \cdot d^*_{\cL_G}(B)} \Big\rfloor
\]
such that $FAB$ is right thick.
\end{theorem}

The question whether the product of two sets of integers with positive upper Banach densities contains
translates of every finite piece of a Bohr set was raised and answered in the paper \cite{BFW} by V. Bergelson, H. Furstenberg and B. Weiss, and was later extended in \cite{BBF} to cover the case of 
products of sets in any countable amenable group by V. Bergelson, M. Beiglb\"ock and the second 
author of this paper. \\

Recall that a set $C \subset G$ is \emph{right piecewise Bohr} if there exists a Bohr set $B \subset G$
and a right thick set $T \subset G$ such that 
\[
C \supset B \cap T. 
\]

\begin{theorem}[\cite{BBF}]
Let $G$ be a countable amenable group and suppose $A, B \subset G$ are $\cL_G$-large
sets. Then $AB$ is a right piecewise Bohr set. In particular, $AB$ is right piecewise left syndetic. 
\end{theorem}

In \cite{G}, J. Griesmer established a beautiful and far-reaching generalization of the main theorem in \cite{BBF}. To 
explain his result, we shall need some notation. Let $G$ be a countable \emph{amenable} group. We 
shall say that a \emph{sequence} $(\nu_n)$ of probability measures is \emph{equidistributed} if for \emph{any} unitary representation $(\cH,\rho)$ of $G$, we have
\[
\forall \, x \in \cH , \quad \lim_n \nu_n(\rho)x = P_{\rho} x,  
\]
in the norm topology on $\cH$, where $P_{\rho}$ denotes the projection onto the $\rho$-fixed vectors, and 
\[
\nu_n(\rho) := \int_{G} \rho(g)  \, d\nu_n(g).
\]
In particular, if $G$ is a countable \emph{abelian} group, then $(\nu_n)$ is equidistributed if and only
if the limit $\hat{\nu}_n(\chi) \ra 0$ holds of all non-zero characters $\chi$ in the dual group of $G$ as $n \ra \infty$. Here, $\hat{\nu}_n$ denotes the Fourier transform of the measure $\nu_n$.\\

We shall say that
a subset $A \subset G$ is \emph{large with respect to $(\nu_n)$} if
\[
\limsup_{n \ra \infty} \nu_n(A) > 0.
\]
Griesmer gives a wealth of examples in \cite{G} of subsets of integers which are not $\cL_{\bZ}$-large, 
but large with respect to some equidistributed sequence. 

One observes that upon considering an equidistributed sequence $(\nu_n)$ as a sequence in the dual of 
$\ell^{\infty}(G)$, then its defining property implies that any accumulation point in the set $\cM(G)$ must 
belong to $\cFS$. In particular, every subset of a countable group which is large with respect to an 
equidistributed sequence must be $\cFS$-large, so the following theorem is strictly contained in our
Theorem \ref{thm2}.

\begin{theorem}[\cite{G}]
Let $G$ be a countable amenable group and suppose $A \subset G$ is large with respect to 
an equidistributed \emph{sequence} of probability measures on $G$ and $B \subset G$ is 
$\cL_G$-large. Then $AB$ is a right piecewise Bohr set. 
\end{theorem}

\subsection{An outline of Theorem \ref{thm1}}

We shall now attempt to break down the proof of Theorem \ref{thm1} into two main steps. The 
proofs of each step will be given in Section \ref{Proof of ergodicity} and 
Section \ref{proof of sum} respectively. We refer to these sections for a more detailed discussion about the notions in this chapter. 

\subsubsection{Ergodicity of $\cC$-densities}

The first main step in the proof of Theorem \ref{thm1} consists of the following generalization of 
Lemma 3.2 in \cite{BBF}. This lemma asserts the same property as the proposition below but for the 
set $\cL_G$ of left invariant means on a countable amenable group instead of the set of 
Furstenberg-Poisson means $\cF_\mu$ on a countable (symmetric) measured group $(G,\mu)$. 

Although not explicitly stated in \cite{BBF}, the only property of $\cL_G$ which is used in the proof 
of Lemma 3.2 in \cite{BBF} is what we shall refer to as \emph{left ergodicity}, namely that every extremal
point of the weak*-compact convex subset $\cL_G \subset \cM(G)$ corresponds (under the 
Gelfand-Naimark map) to a left $G$-invariant \emph{ergodic} probability measure on $\beta G$. 

This property is relatively straightforward to verify for the sets $\cL_G$ (when $G$ is amenable) and
$\cL_\mu$ (for a measured group $(G,\mu)$). However, it also holds for the set $\cF_\mu$ but this fact
requires 
a more detailed discussion about the Furstenberg-Poisson boundary of $(G,\mu)$ which will be
outlined in Subection \ref{poisson}.
  
\begin{proposition}
\label{ergodicity}
Suppose $A \subset G$ is $\cF_\mu$-large. Then,
\[
\sup d^*_{\cF_\mu}(FA) = 1,
\]
where the supremum is taken over all \emph{finite} subsets $F \subset G$.
\end{proposition}

\subsubsection{A criterion for right thickness of product sets}

The second main step is inspired by Lemma 3.1 in \cite{BBF}, and indeed reduces to this 
lemma in the case when $\cL_G = \cL_\mu$. However, the proof of this step for a general 
measured group necessarily follows a different route than the proof given in \cite{BBF}. 
For instance, we do not know if the proposition holds with $\cF_\mu$ replaced with $\cL_\mu$ 
in the first addend. 

\begin{proposition}
\label{sum}
Suppose $A, B \subset G$ satisfy
\[
d^*_{\cF_\mu}(A) + d^*_{\cL_\mu}(B) > 1.
\]
Then $AB$ is right thick. 
\end{proposition}

We shall now combine the two propositions above to give a proof of Theorem \ref{thm1}.

\subsubsection{Proof of Theorem \ref{thm1}} 
Suppose $A \subset G$ is $\cF_\mu$-large and $B \subset G$ is $\cL_\mu$-large. 
We wish to prove that there exists a finite set $F \subset G$ such that the triple 
product set $FAB$ is right thick. By Proposition \ref{sum}, it suffices to prove that
\[
d^*_{\cF_\mu}(FA) + d^*_{\cL_\mu}(B) > 1
\]
for some finite set $F \subset G$, or equivalently,
\[
\sup_F d^*_{\cF_\mu}(FA) > 1 - d^*_{\cL_\mu}(B),
\]
where the supremum is taken over all finite subsets $F \subset G$. The result now
follows from Proposition \ref{ergodicity}.

\subsection{An outline of Theorem \ref{thm2}}

The proof of Theorem \ref{thm2} can be divided into three main steps, where the first
step is the most important one, and the two other steps are more of a standard nature. 
However, since we do not know of a good reference for these steps in the generality 
needed here, we shall provide proofs of the necessary statements.

\subsubsection{Producing some Bohr structure}

The first step will be outlined in Section \ref{Bohr}. For definitions and a more detailed discussion
about the notation in this summary we refer the reader to Subsection \ref{measure theory}.\\

Let $G$ be a countable group and let $\beta G$ denote its Stone-\v{C}ech compactification. 
Given a subset $C \subset G$, we denote by $\overline{C}$ its "closure" in $\beta G$. 
If $q \in \beta G$, then we shall write
\[
\oC_q := \Big\{ g \in G \, : \, g \cdot q \in \oC \Big\} \subset G.
\] 
In particular, if $\bar{e}$ denotes the "image" of the identity element $e$ in $\beta G$,
then $C = \oC_{\bar{e}}$. \\

\begin{proposition}
\label{bohr appears}
Suppose $A \subset G$ is $\cFS$-large and $B \subset G$ is strongly non-paradoxical. 
Then there exist Borel sets $\tilde{A}, \tilde{B} \subset bG$ with
\[
m(\tilde{A}) \geq d^*_{\cFS}(A) \qqand m(\tilde{B}) \geq d^*_{np}(B)
\]
and $q \in \beta G$ and $k_o \in bG$ such that 
\[
\overline{AB}_q \supset \iota^{-1}(U \cdot k_o),
\]
where
\[
U := \Big\{ k \in bG \, : \, m(\tilde{A} \cap k \cdot \tilde{B}^{-1}) > 0 \Big\} \subset bG.
\]
\end{proposition}

\subsubsection{Piecewise inclusion}

Let $B, C \subset G$ and let $\oB, \oC \subset \beta G$ denote the corresponding
closures in $\beta G$. Suppose we have an inclusion of the form
\[
\oC_q \supset \oB_q,
\]
for some $q \in \beta G$. The next proposition shows that if 
$\oB$ is the pull-back to $\beta G$ (under a continuous $G$-equivariant map)
of a \emph{Jordan measurable} subset of the Bohr compactification $\bG$, then 
some "virtual" translate of $B$ is "almost" completely contained in $C$. 
Recall that a set $U \subset \bG$ is \emph{Jordan measurable} if its 
topological boundary $\overline{U} \setminus U^o$ has zero Haar measure.

\begin{proposition}
\label{local inclusion}
Let $C \subset G$ and suppose there exist $q \in \beta G$ and an open Jordan measurable
set $U \subset bG$ such that
\[
\overline{C}_q \supset \iota^{-1}(U).
\]
Then there exist a right thick set $T \subset G$ and $k \in bG$ such that
\[
C \supset \iota^{-1}(U \cdot k) \cap T. 
\]
\end{proposition}

\subsubsection{Getting Jordan measurability}

The last step is the most standard one and is needed to adapt the outcome of 
Proposition \ref{bohr appears} to the setting of Proposition \ref{local inclusion}. 
The main result here is a generalization to compact Hausdorff groups of the 
well-known fact that every \emph{closed} subset of a compact \emph{metrizable}
group is contained in a Jordan measurable subset with Haar measure arbitrarily close to
the Haar measure of the closed set. The arguments needed to relate this generalization
to the proposition below will be outlined in Section \ref{Jordan}.

\begin{proposition}
\label{jordan}
Let $K$ be a compact Hausdorff group with Haar probability measure $m$. If $A, B \subset K$ 
are Borel sets with positive Haar measures, and
\[
U := \Big\{ k \in K \, : \, m(A \cap k \cdot B) > 0 \Big\} \subset K,
\]
then there exists an open Jordan measurable set $U' \subset U$ with
\[
s_K(U') \leq \Big\lfloor \frac{1}{m(A) \cdot m(B)} \Big\rfloor.
\]
\end{proposition}

We shall now combine the three propositions above to give a proof of Theorem \ref{thm2}.

\subsubsection{Proof of Theorem \ref{thm2}}
Let $G$ be a countable group and suppose $A \subset G$ is $\cFS$-large and $B \subset G$ is 
strongly non-paradoxical. \\

By Proposition \ref{bohr appears} there exist Borel sets $\tilde{A}, \tilde{B} \subset bG$ with
\[
m(\tilde{A}) \geq d^*_{\cFS}(A) \qqand m(\tilde{B}) \geq d^*_{np}(B)
\]
and $q \in \beta G$ and $k_o \in bG$ such that 
\[
\overline{AB}_q \supset \iota^{-1}(U \cdot k_o),
\]
where
\[
U := \Big\{ k \in bG \, : \, m(\tilde{A} \cap k \cdot \tilde{B}^{-1}) > 0 \Big\} \subset bG.
\]
Proposition \ref{jordan} guarantees that we can find an open and \emph{Jordan measurable}
subset $U' \subset U$ with
\[
s_{bG}(U') 
\leq 
\Big\lfloor \frac{1}{m(\tilde{A}) \cdot m(\tilde{B}^{-1})} \Big\rfloor 
\leq 
\Big\lfloor \frac{1}{d^*_{\cFS}(A) \cdot d^*_{\textrm{np}}(B)} \Big\rfloor,
\]
where the last inequality is justified by the fact that $m(\tilde{B}^{-1}) = m(\tilde{B})$. Now, 
\[
\overline{AB}_q \supset \iota^{-1}(U \cdot k_o) \supset \iota^{-1}(U' \cdot k_o),
\]
so Proposition \ref{local inclusion} applies with $C = AB$, and we conclude that there exist a 
right thick set $T \subset G$ and $k \in bG$ such that
\[
AB \supset \iota^{-1}(U' \cdot k_o \cdot k) \cap T,
\]
which finishes the proof.

\subsection{Organization of the paper}
The paper is organized as follows. Section \ref{preliminaries} is devoted to basic facts about
the C*-algebraic theory of sub-algebras of the space of bounded functions on a countable 
group $G$, and to some basic measure theory on $\beta G$. In particular, we give 
characterizations of right thick sets and left syndetic sets in terms of the upper and lower $\cC$-densities
for special classes of sets $\cC \subset \cM(G)$.

In Section \ref{ergodic theory} we develop some basic ergodic theory for measured groups which will 
be needed in the proofs of Proposition \ref{ergodicity} and Proposition \ref{sum}. In the final 
subsection of this section we outline a proof of claim about failure of syndeticity of difference sets in general countable measured groups.  

In Section \ref{Proof of ergodicity} and Section \ref{proof of sum} we give the proofs of the assertions
needed in the proof of Theorem \ref{thm1}. In Section \ref{Bohr} we discuss unitary representations of
countable groups; in particular we make the direct sum of the Eberlein algebra more explicit. This is 
needed for the proof of Proposition \ref{bohr appears}.

In Section \ref{Local} we discuss the relevance of almost automorphic points to 
Proposition \ref{local inclusion}, and give a proof of this proposition. 

In Section \ref{Jordan} we prove some elementary facts about Jordan measurable subsets of 
compact Hausdorff groups needed for Proposition \ref{jordan}.

\subsection{Acknowledgements}
The authors have benefited enormously from discussions with Benjamin Weiss 
during the preparation of this manuscript, and it is a pleasure to thank him for 
sharing his many insights with us. We are also very grateful for the many inspiring
and enlightening discussions we have had with Mathias Beiglb\"ock, Vitaly Bergelson,
Manfred Einsiedler, Hillel Furstenberg, Eli Glasner, John Griesmer, Elon Lindenstrauss, 
Fedja Nazarov, Amos Nevo, Imre Ruzsa and Jean-Paul Thouvenot. 

The present paper is an outgrowth of a series of discussions between the authors 
which took place at Ohio State University, University of Wisconsin, Hebrew University
in Jerusalem, Weizmann institute, IHP Paris, KTH Stockholm and ETH Z\"urich. We 
thank the mathematics departments at these places for their great hospitality. 

The first author acknowledges support from the European Community seventh 
Framework program (FP7/2007-2012) grant agreement 203418 when he was 
a postdoctoral fellow at Hebrew university, and ETH Fellowship FEL-171-03
since January 2011.

\section{Preliminaries and basic notions}
\label{preliminaries}

\subsection{Uniform algebras}
We shall now review some basic elements in the theory of norm closed sub-algebras of 
$\ell^{\infty}(G)$. For more details, including the proofs of the all statements in this subsection, 
we refer the reader to the second chapter of  the book \cite{Ka}; In particular, subsections $2.9$ 
and $2.10$ therein. 

\subsubsection{The Stone-\v{C}ech compactification}
Let $G$ be a countable group, and let $\ell^{\infty}(G)$ denote the algebra of bounded functions 
on $G$. We denote the complex conjugation operator on $\ell^{\infty}(G)$ by $*$. Equipped with this involution
and the supremum norm $\|\cdot\|_\infty$, this algebra carries the structure of 
a unital commutative C*-algebra. As such, its Gelfand spectrum $\beta G$, i.e. the set of multiplicative 
functionals on $\ell^{\infty}(G)$, endowed with the weak*-subspace topology inherited from the 
unit ball of the dual of $\ell^{\infty}(G)$, is a weak*-compact Hausdorff space. Moreover, by 
Gelfand-Naimark's Theorem, we have 
\[
\ell^{\infty}(G) \cong C(\beta G),
\]
where $C(\beta G)$ denotes the space of continuous function on $\beta G$. 

Furthermore, the left and right action of $G$ on itself induce left and right anti-actions on $\ell^{\infty}(G)$, and thus left and right actions by homeomorphisms on $\beta G$. We shall refer to $\beta G$ as 
the \emph{Stone-\v{C}ech compactification} of $G$.

The map $\delta_e$ which evaluates a bounded function on $G$ at the identity element $e$  is a multiplicative functional on $\ell^{\infty}(G)$ and thus corresponds to a point $\bar{e}$ in $\beta G$. One readily shows that this point is $G$-transitive under both the left and right action of $G$ on $\beta G$. Moreover, a left $G$-equivariant isomorphism between $C(\beta G)$ and $\ell^{\infty}(G)$ is implemented by the map
\begin{equation}
\label{sc}
\tau(\varphi)(g) = \varphi(g \cdot \bar{e}), \quad \varphi \in C(\beta G). 
\end{equation}

\subsubsection{Relations with sub-algebras}

More generally, if $\cB \subset \ell^{\infty}(G)$ is a norm-closed unital sub-algebra of $\ell^{\infty}(G)$, then it is again a C*-algebra, and its Gelfand spectrum $\Delta(\cB)$ is a compact Hausdorff space with the property that
\[
\cB \cong C(\Delta(\cB)).
\]
By Urysohn's Lemma, $\Delta(\cB)$ is second countable if and only if $\cB$ is separable. \\

If $\cB$ is invariant under the left anti-action of $G$ on $\ell^{\infty}(G)$, then this anti-action induces 
an action of $G$ by homeomorphisms on $\Delta(\cB)$. For every $g \in G$, the Dirac measure 
$\delta_g$ at $g$ is a multiplicative functional on $\cB$. Let $x_o$ denote the point in $\Delta(\cB)$ corresponding to the evalutation map $\delta_e$ and let $X$ denote the orbit closure of $x_o$ under the induced $G$-action. 
If $X \neq \Delta(\cB)$, then the complement of $X$ is open and we can find a non-zero continuous function $\overline{\varphi}$ on $\Delta(\cB)$  (corresponding to a non-zero element $\varphi$ in $\cB$) which vanishes on $X$. In particular, 
\[
\forall \, g \in G, \quad g \cdot x_o(\varphi) = \varphi(g) = 0,  
\]
which is a contradiction, since $\varphi$ was assumed non-zero, and thus $X = \Delta(\cB)$. We conclude that $x_o$ is a $G$-transitive point in $\Delta(\cB)$. By Gelfand-Naimark's Theorem, the inclusion map $\cB \hookrightarrow \ell^{\infty}(G)$ is induced by a 
surjective continuous map $\pi_{\cB} : \beta G \ra \Delta(\cB)$, i.e. $\iota = \pi_{\cB}^*$, where
\[
\pi_{\cB}^{*}\varphi := \varphi \circ \pi_{\cB}. 
\] 
If $\cB$ is left invariant, then this map is left $G$-equivariant (with respect to the left action on $\beta G$) and 
$\pi_{\cB}(\bar{e}) = x_o$.

\subsubsection{Almost periodic functions}
\label{Bohr stuff}
We shall now consider a few examples of the situation outlined in the previous paragraph. 

Let $\AP(G)$ denote the algebra of \emph{almost periodic functions} on $G$, i.e. the $*$-sub-algebra of 
$\ell^{\infty}(G)$ consisting of those bounded functions $\varphi$ whose left anti-$G$-orbit $G \cdot \varphi$ in 
$\ell^{\infty}(G)$ is pre-compact in the norm topology. Equipped with the conjugation involution and the supremum norm, this $*$-algebra is a sub-C*-algebra of $\ell^{\infty}(G)$, and we shall denote its Gelfand spectrum by $bG$. 

We refer to subsection $2.10$ in the book \cite{Ka} for the proof that the group multiplication on $G$ 
extends to a jointly continuous multiplication on $bG$, i.e. $bG$ can be given the structure of a 
compact Hausdorff group. Again, there is a natural (but this time not necessarily injective) homomorphism $\iota : G \ra bG$ given by 
\[
g \mapsto \delta_g,
\] 
where $\delta_g$ denotes the evaluation map on $\ell^{\infty}(G)$ at the element $g$. Furthermore, this map is universal in the sense that whenever $K$ is a compact Hausdorff group and $\iota_o : G \ra K$ is a homomorphism with a dense image,
then there exists a \emph{continuous} surjective homomorphism $\iota_K : bG \ra K$ such that
\[
\iota_K \circ \iota = \iota_o.
\]
We shall refer to $(bG, \iota) $ as the \emph{Bohr compactification} of $G$, and the \emph{unique} Haar probability measure on $bG$ will be denoted by $m$. Note that by construction, we have
\[
\AP(G) = \iota^*C(\beta G). 
\]
Recall that a state (or mean) on $\ell^{\infty}(G)$ is a positive, norm-one functional on $\ell^{\infty}(G)$. The
set $\cM(G)$ of all states on $\ell^{\infty}(G)$ is a weak*-compact and convex subset of $\ell^{\infty}(G)^*$. 
We shall say that $\lambda \in \cM(G)$ is a \emph{Bohr mean} if the restriction of $\lambda$ to $\AP(G)$
equals the pull-back of the Haar probability measure on $bG$ and denote the \emph{non-empty} weak*-compact and convex subset of $\cM(G)$ consisting of all Bohr means by $\cBo$.

\subsubsection{Fourier-Stiljes algebras}

Let $\B(G)$ denote the \emph{Fourier-Stiltjes algebra} of $G$, i.e. the $*$-sub-algebra of 
$\ell^{\infty}(G)$ generated by all matrix coefficients of \emph{unitary} representations of $G$. 
Hence an element $\varphi$ in $\B(G)$ has the form
\[
\varphi(g) = \big\langle x, \pi(g) y \big\rangle, \quad g \in G,
\]
where $(\cH,\pi)$ is some unitary representation of $G$ on a Hilbert space $\cH$, and $x, y \in \cH$. We
stress that $\B(G)$ is not a norm-closed subspace of $\ell^{\infty}(G)$. Its norm closure, often called the 
\emph{Eberlein algebra}, will here be denoted by $\E(G)$, and its Gelfand spectrum will be denoted by 
$eG$. 

As is well-known (see e.g. Chapter 1.10 in \cite{Gl}), there is a unique left and right invariant state 
$\lambda_o$ on $\E(G)$, and $\E(G)$ admits a direct sum decomposition of the form
\begin{equation}
\label{decomposition}
\E(G) = \AP(G) \oplus \E_o(G),
\end{equation}
where
\[
\E_o(G) := \Big\{ \varphi \in \E(G) \, : \, \lambda_o(\varphi^*\varphi) = 0 \Big\}.
\]
The set $\cFS \subset \cBo(G)$ consisting of those means whose restrictions to $\E(G)$ coincide with
$\lambda_o$ will be referred to as \emph{Fourier-Stiltjes means}.

\subsection{Bauer's Maximum Principle}
Let $G$ be a countable group and let $\cM(G)$ denote the set of states on $\ell^{\infty}(G)$, or alternatively, the set of means on $G$. If 
$\lambda \in \cM(G)$, then we define
\[
\lambda'(C) := \lambda(\chi_C), \quad C \subset G.
\]
One readily checks that $\lambda'$ is a finitely additive (but not necessarily $\sigma$-additive) probability measure on $G$ and for every $C \subset G$, the \emph{affine} map
\[
\lambda \mapsto \lambda'(C), \quad \lambda \in \cM(G)
\] 
is weak*-continuous.\\

Given a set $\cC \subset \cM(G)$, we define the \emph{upper} and \emph{lower} $\cC$-densities of a subset $C \subset G$ by
\[
d^*_{\cC}(C) := \sup_{\lambda \in \cC} \lambda'(C) 
\qand
d^{\cC}_*(C) := \inf_{\lambda \in \cC} \lambda'(C) 
\]
respectively. If $\cC$ is weak*-compact, then the supremum and infimum above are attained. Furthermore,
since the maps involved are affine, we can also say some things about the location of these extrema, as the following result shows.

Before we state the main proposition of this section (which is known as Brauer's Maximum Principle), we recall that the classical Krein-Milman's Theorem asserts that any non-empty weak*-compact and convex subset $\cC$ of the dual of a Banach space has extremal points, i.e. elements $\lambda$ in $\cC$ which do not admit a representation of the form
\[
\lambda = \alpha_1 \cdot \lambda_1 + \alpha_2 \cdot \lambda_2,
\]
where $\lambda_1, \lambda_2 \in \cC$ and $\alpha_1, \alpha_2$ are \emph{positive} numbers with
$\alpha_1 + \alpha_2 = 1$. 

\begin{proposition}[Theorem 7.69 in \cite{AB06}]
Let $\cC$ be a weak*-compact convex subset of the dual of a Banach space and suppose 
$\varphi : \cC \ra \bR$ is an upper semicontinuous convex function. Then $\varphi$
attains its maximum at an extremal point in $\cC$. 
\end{proposition}

Applied to the Banach space $X = \ell^{\infty}(G)$, we have the following immediate corollary,
which shall be used frequently. 

\begin{corollary}
\label{max2}
Let $X$ be a Banach space and suppose $\cC \subset X^*$ is a weak*-compact convex set. 
For every $x \in X$, there exist an extremal $\lambda_o \in \cC$ such that
\[
\lambda_o(x) = \sup_{\lambda \in \cC} \lambda(x).
\]
In particular, for any weak*-compact convex set $\cC \subset \cM(G)$ and for any $C \subset G$, there exist extremal elements $\lambda_o$ and $\lambda_1$ in $\cC$ such that
\[
d_{\cC}^*(C) = \lambda_o'(C)
\qqand
d^{\cC}_*(C) = \lambda_1'(C).
\]
\end{corollary}

For separable Banach spaces, this corollary can be deduced from Choquet's Theorem. However, 
in the case $X$ is \emph{not} separable (which is of crucial interest to us), there is no reason to expect 
that the set of extremal points of $\cC \subset X^*$ is even Borel measurable, and thus disintegration 
techniques are not directly applicable. 

\subsection{Measure theory on $\beta G$}
\label{measure theory}

There is an important interchange between subsets of the group $G$ and \emph{clopen} subsets 
of $\beta G$ which we now describe.

\subsubsection{Set correspondences}
Recall that the map $\tau : C(\beta G) \ra \ell^{\infty}(G)$ given by
\[
\tau(f)(g) = f(g \cdot \bar{e}), \quad g \in G,
\]
is an isometric $*$-isomorphism. Hence, if $C \subset G$, then there exists a \emph{unique} 
$\varphi$ in $C(\beta G)$ such that $\chi_C = \tau(\varphi)$,
and thus 
\[
\tau(\varphi^* \varphi) = \tau(\varphi)^* \tau(\varphi) = \chi_C^2 = \tau(\varphi),
\]
which implies that $\varphi$ is idempotent in $C(\beta G)$. Since $\tau$ is an isomorphism, this
entails that  $\varphi$ is necessarily equal to the indicator function of a subset $\oC$ of $\beta G$. 
Since the indicator function of this set is continuous, we conclude that $\oC$ is clopen. Conversely, 
if $\oC$ is a clopen set, then one readily checks that $\tau(\chi_{\oC})$ is the
indicator function of a subset $C$ of $G$. Hence we have a one-to-one correspondence 
\[
C \longleftrightarrow \oC
\] 
between subsets of $G$ and \emph{clopen} subsets of $\beta G$. 

\subsubsection{Measure correspondences}

Since the map $\tau$ above is clearly positivity preserving and unital, we see that if $\lambda$
is a mean on $G$, then $\tau^*\lambda$ is a positive unital continuous functional on $C(\beta G)$,
and can thus be thought of as a  $\sigma$-additive Borel probability measure on $\beta G$, via 
Riesz Representation Theorem. Conversely, any Borel probability measure $\olambda$ on 
$\beta G$ induces a linear functional on $C(\beta G)$ by integration, and thus a state 
on $\ell^{\infty}(G)$ (or equivalently, a mean on $G$). Hence there is a one-to-one correspondence
between means on $G$ and probability measures on $\beta G$ given by
\[
\lambda \longleftrightarrow \olambda.
\] 

\subsubsection{Beyond continuous functions}

Note that the map $\tau$ above is clearly defined for \emph{all} functions on $\beta G$, continuous 
or not. However, easy examples show that there may be no relation between the $\olambda$-integral
of a discontinuous function $\varphi$ on $\beta G$ and the $\lambda$-value of $\tau(\varphi)$ for a 
given mean $\lambda$ on $G$. 

However, as we shall see in the following lemma, it is 
possible to relate the two values when $\varphi$ is a \emph{lower semicontinuous} function on $\beta G$. 

\begin{lemma}
\label{lower bound help1}
Suppose that a bounded positive function $\varphi$ on $\beta G$ equals the supremum of an increasing sequence of continuous functions on $\beta G$. Then
\[
\lambda(\tau(\varphi)) \geq \int_{\beta G} \varphi \, d\olambda, \quad \forall \, \lambda \in \cM(G).
\]
\end{lemma}

\begin{proof}
Let $(\varphi_n)$ be an increasing sequence of continuous functions on $\beta G$ such that
\[
\varphi = \sup_n \varphi_n.
\]
Fix $\lambda \in \cM(G)$. Then, for every $n$, we have
\[
\lambda(\tau(\varphi)) \geq \lambda(\tau(\varphi_n)) = \int_{\beta G} \varphi_n \, d\olambda.
\]
Since $\olambda$ is a $\sigma$-additive Borel probability measure on $\beta G$, 
\[
\sup_n \int_{\beta G} \varphi_n \, d\olambda = \int_{\beta G} \varphi \, d\olambda
\]
by monotone convergence, which finishes the proof.
\end{proof}

\subsubsection{Integrating open sets in $\beta G$}

Let us recall the notion of a \emph{lower} $\cC$-density for a countable
group $G$. If 
$\cC \subset \cM(G)$, we defined the \emph{lower $\cC$-density} of a subset 
$C \subset G$ by
\[
d_*^{\cC}(C) := \inf_{\lambda \in \cC} \lambda'(C).
\]
Now, if $B \subset \beta G$ is any Borel set and $q$ is any point in $\beta G$, 
we can construct "the return time" subset $B_q \subset G$ by
\[
B_q := \Big\{ g \in G \, : \, g \cdot q \in B \Big\}.
\]
In particular, if $C \subset G$ and $\oC$ denotes the corresponding clopen set in 
$\beta G$, then $C = \oC_{\bar{e}}$, and 
\[
\lambda'(\oC_{\bar{e}}) = \olambda(\oC), \quad \forall \, \lambda \in \cM(G).
\]
When $B$ is no longer clopen in $\beta G$, there is no reason to expect any 
relation of this sort. However, as we shall see in the next lemma, there is a 
very important lower bound, in the case when $B$ consists of \emph{countable}
union of clopen sets.  

\begin{lemma}
\label{lower bound help2}
Let $\cC \subset \cM(G)$ be a weak*-compact and convex set and suppose $U \subset \beta G$ is a countable union of clopen sets in $\beta G$. Then there exists an extremal $\lambda \in \cC$ such
that
\[
d_{*}^{\cC}(U_{\bar{e}}) \geq \olambda(U).
\]
\end{lemma}

\begin{proof}
By Corollary \ref{max2}, applied to the set $C := U_{\bar{e}}^c$, there exists an extremal 
$\lambda \in \cC$ such that
\[
\lambda'(U_{\bar{e}}) = d_*^{\cC}(U_{\bar{e}}).
\]
By assumption, the indicator function $\varphi := \chi_U$ equals the supremum of an increasing sequence
of indicator functions on clopen sets (which are continuous), so by Lemma \ref{lower bound help1}, we 
have
\[
\lambda'(U_{\bar{e}}) = \lambda(\tau(\chi_{U})) \geq \olambda(U),
\]
which finishes the proof. 
\end{proof}

\subsection{Combinatorics on $\beta G$}

The next two subsections will be concerned with the problem of reading off combinatorial properties 
of a subset $C \subset G$ from topological properties of $\overline{C} \subset \beta G$ and 
dynamical properties of the "return time sets"
\[
\oC_q = \Big\{ g \in G \, : \, g \cdot q \in \oC \Big\} \subset G,
\]
for various choices of $q \in \beta G$. 

\begin{lemma}
\label{UF}
For every $C \subset G$ and finite $F \subset G$, the sets
\[
U_F(C) := \Big\{ q \in \beta G \, : \, F \subset \oC_q \Big\}
\qand
U^F(C) := \Big\{ q \in \beta G \, : \, \oC_q \subset G \setminus F \Big\}
\]
are clopen subsets of $\beta G$. 
\end{lemma}

\begin{proof}
Since $\oC \subset \beta G$ is clopen and
\[
U^F(C) = U_F(C^c),
\]
it suffices to prove that $U_F(C)$ is clopen in $\beta G$ for every $C \subset G$ and finite $F \subset G$.
Note that
\[
U_F(C) = \bigcap_{f \in F} U_{\{f\}}(C),
\]
and
\[
\forall \, f \in F, \quad U_{\{f\}}(C) = \Big\{ q \in \beta G \, : \, f \in \oC_q \Big\} = f^{-1} \cdot \oC. 
\]
Since the latter sets are all clopen and $U_F(C)$ is a finite intersection of such sets, we conclude that
$U_F(C)$ is clopen as well. 
\end{proof}

\begin{lemma}
\label{UF2}
For every $C \subset G$, $q \in \beta G$ and finite $F \subset G$, there exists $g \in G$
such that
\[
C \cdot g^{-1} \cap F = \oC_{q} \cap F.
\]
In fact, for every $q \in \beta G$, the set of $p \in \beta G$ such that 
\[
\oC_{p} \cap F = \oC_{q} \cap F
\]
is clopen. 
\end{lemma}

\begin{proof}
Fix $q \in \beta G$ and a finite set $F \subset G$ and define
\[
I := F \cap \oC_q \qand J := F \cap (\oC_q)^c
\]
By Lemma \ref{UF} the set
\[
E_F(q) :=  
\Big\{ p \in \beta G  \, : \, \oC_p \cap F = \oC_q \cap F \Big\}
=
U_{I}(C) \cap U^{J}(C)
\]
is clopen and clearly contains $q$. Since $\bar{e} \in \beta G$ is $G$-transitive, we can find
$g \in G$ such that $g \cdot \bar{e} \in E_F(q)$, or equivalently,
\[
\oC_{g \cdot \bar{e}} \cap F = C \cdot g^{-1} \cap F = \oC_q \cap F,
\]
which finishes the proof. 
\end{proof}

\subsubsection{Jordan measurability and local inclusion}

Recall that if $\nu$ is a Borel probability measure on a compact Hausdorff space $X$, then a 
Borel set $C \subset X$ is \emph{$\nu$-Jordan measurable} if the $\nu$-measure of the 
\emph{topological boundary}
\[
\partial B := \oB \setminus B^o \subset X
\]
is zero, where $B^o$ denotes the interior of $B$. 

In many of the applications in this paper, we shall require a stronger form of Jordan measurability. 
Let $\cP(X)$ denote the space of Borel probability measures on $X$, identified with the set of states
of the C*-algebra $C(X)$ via the Riesz Representation Theorem. Given a subset $\cC \subset \cP(X)$
we shall say that a set $C \subset X$ is \emph{$\cC$-tempered} if $C$ is $\nu$-Jordan measurable for 
\emph{all} $\nu \in \cC$.  

The property of $\cC$-temperedness behaves well under liftings as the following lemma shows. 

\begin{lemma}
\label{lift}
Let $\cC \subset \cP(X)$ and suppose $\pi : X \ra Y$ is a continuous map. 
If $B \subset Y$ is $\pi_*\cC$-tempered, then $\pi^{-1}(B) \subset X$ is 
$\cC$-tempered. 
\end{lemma}

\begin{proof}
This is an immediate consequence of the inclusions
\[
\pi^{-1}(B^o) \subset \pi^{-1}(B)^o \subset \pi^{-1}(B) \subset \overline{\pi^{-1}(B)} \subset \pi^{-1}(\overline{B}),
\]
and the monotonicity of positive measures. 
\end{proof}

The next result will be a crucial ingredient in the proof of Proposition \ref{local inclusion}. A key point 
in the proof is the \emph{extreme disconnectedness} of $\beta G$. Recall that a compact Hausdorff space is 
\emph{extremely disconnected} if the closure of any open subset is \emph{clopen}. The classical 
fact that $\beta G$ is extremely disconnected for \emph{any} discrete group $G$ can be found in \cite{Jo}.\\

Also recall that a set $\cC \subset \cM(G)$ is \emph{left saturated} if 
every closed left-$G$-invariant subset $Y \subset \beta G$ supports
a probability measure of the form $\olambda$ for some $\lambda$ in 
$\cC$ with the property that whenever $A \subset \beta G$ is a Borel
set with $\olambda(A) = 0$, then $\olambda(g \cdot A) = 0$ for all 
$g \in G$.

\begin{lemma}
\label{local inclusion0}
Let $\cC \subset \cM(G)$ be a left saturated set and denote by $\overline{\cC}$ its image in $\cP(\beta G)$.
Let $C \subset G$ and suppose $U \subset \beta G$ is an open $\overline{\cC}$-tempered set. If 
\[
\oC_q \supset U_q
\]
for some $q \in \beta G$, then there exists a $\cC$-conull set $T \subset G$ such that
\[
C \supset U_{\bar{e}} \cap T.
\]
\end{lemma}

\begin{proof}
Define the open set $D := U \setminus \oC$. By assumption, $D_q$ is empty, and thus
\[
Y := \beta G \setminus GD \subset \beta G
\]
is a non-empty, closed and $G$-invariant set. Since $\cC$ is left saturated, there exists
$\lambda \in \cC$ such that $\olambda(Y) = 1$, and thus $\olambda(D) = 0$. Since 
$\oC$ is clopen and $U$ is Jordan measurable with respect to $\olambda$, we have
\[
\olambda(\oU) = \olambda(U) = \olambda(U \cap \oC) = \olambda(\oU \cap \oC),
\]
and hence $\olambda(\oU \setminus \oC) = 0$. Since $\beta G$ is extremely disconnected and
$U$ is open, $\oU \setminus \oC$ is clopen, and thus
\[
\lambda((\oU \setminus \oC)_{\bar{e}}) = \olambda(\oU \setminus \oC).
\]
Define the sets
\[
\oT := \beta G \setminus \big( \oU \setminus \oC \big) \qand T := \oT_{\bar{e}}.
\]
Then $\lambda(T) = \olambda(\oT) = 1$, so $T$ is $\cC$-conull, and
\[
\oC_{\bar{e}} \supset \oU_{\bar{e}} \cap T \supset U_{\bar{e}} \cap T,
\]
which finishes the proof. 
\end{proof}

\subsection{Properties of right  thick sets}

Recall that a subset $T \subset G$ is \emph{right thick} if for \emph{every} finite subset $F \subset G$
there exists $g \in G$ such that
\[
F g^{-1} \subset T.
\]
The following lemma is a simple characterization of right thickness in a countable group $G$ 
in terms of clopen subsets of $\beta G$.

\begin{lemma}
\label{thick lemma}
A set $T \subset G$ is right thick if and only if $\oT_q = G$ for some $q \in \beta G$. 
\end{lemma}

\begin{proof}
Suppose $\oT_q = G$ for some $q \in \beta G$ and fix a finite set $F$ in $G$. By Lemma
\ref{UF2}, there exists $g \in G$ such that
\[
T \cdot g^{-1} \supset F,
\] 
and hence $T$ is right thick.

Now suppose that $T$ is right thick and choose an increasing sequence $(F_n)$ of finite sets with 
$\cup_n F_n = G$. Since $T$ is right thick, we can find a sequence $(g_n)$ in $G$ such
that 
\[
T = \oT_{\bar{e}} \supset F_n \cdot g_n^{-1},
\]
or equivalently $\oT_{g_n^{-1} \cdot \bar{e}} \supset F_n$, for all $n$. Extract a convergent
subset $(g_{n_\alpha}^{-1} \cdot \bar{e})$ with limit point $q$. We claim that $\oT_q = G$. Indeed, fix any 
finite subset $F \subset G$. By Lemma \ref{UF}, the set $U_F(T)$ is clopen. Note that 
$g_{n_\alpha} \cdot \bar{e} \in U_F(T)$ for all $n_\alpha$ with $F \subset F_{n_\alpha}$. In particular, 
$q \in U_F(T)$ since $U_F(T)$ is closed. Now, $F$ was chosen arbitrary, so we conclude that $\oT_q = G$.
\end{proof}

\subsubsection{Density characterizations of right thick sets}

If $G$ is a countable \emph{amenable} group, then right thickness has a simple description in terms of 
invariant means. Namely, suppose $T \subset G$ is right thick and choose a left F\o lner sequence 
$(F_n)$ in $G$. Since $T$ is right thick, we can find $(g_n)$ such that
\[
F_n \cdot g_n^{-1} \subset T, \quad \forall \, n.
\] 
Note that the sequence $F_n' := F_n g_n^{-1}$ is still left F\o lner, so any left invariant mean constructed
as a weak*-accumulation point of the averages over $(F_n')$ will give measure one to the set $T$. In 
our terminology, $T$ is $\cL_G$-conull, where $\cL_G$ denotes the set of left-invariant means on $G$. It is not hard to establish the converse as 
well; that is to say, a subset of a countable amenable group is 
right thick if and only if it is $\cL_G$-conull. 

Recall that a set $\cC(G) \subset \cM(G)$ is left saturated if every closed $G$-invariant
subset of $\beta G$ supports a probability measure of the form $\olambda$ for some $\lambda$ in $\cC$ which has the property that
if $A \subset \beta G$ is a Borel set with $\olambda(A) = 0$, then 
$\olambda(g \cdot A) = 0$ for all $g \in G$. Note that amenability
of a countable group $G$ is equivalent to $\cL_G$ being left saturated.

\begin{proposition}
\label{thick help}
Suppose $\cC \subset \cM(G)$ is a left saturated set. Then $T \subset G$ is right thick if and only if it
is $\cC$-conull, i.e. $d^*_{\cC}(T) = 1$.
\end{proposition}

\begin{proof}
Suppose $T$ is right thick. By Lemma \ref{thick lemma} there exists $q \in \beta G$ such that 
$\oT_q = G$. Let $Z$ denote the orbit closure of $q$ in $\beta G$ under the left action by
$G$ on $\beta G$. Note that $\oT_q = G$ implies that $Z \subset \oT$. Since $\cC$ is 
left saturated, we can find $\lambda \in \cC$ with $\olambda(Z) = 1$, and thus 
$\lambda(T) = 1$, which shows that $T$ is $\cC$-conull. 

Now suppose that $T$ is $\cC$-conull and fix $\lambda \in \cC$ with $\olambda(\oT) = 1$.
Since $\cC$ is left saturated, $\olambda$ is non-singular under the left action of $G$ on $\beta G$.
In particular, the support $Z$ of $\olambda$ is a closed $G$-invariant set, and $Z \subset \oT$.
Hence, for all $q \in Z$, we have $\oT_q = G$, which by Lemma \ref{thick lemma} implies that $T$
is right thick. 
\end{proof}

\subsection{Properties of left syndetic sets}

Recall that a set $C \subset G$ is \emph{left syndetic} if there exists a finite set $F \subset G$ such
that $FC = G$. The following lemma isolates an important property of these sets.

\begin{lemma}
\label{syndetic lemma}
A set $C \subset G$ is left syndetic if and only if $\oC_q$ is non-empty for all $q \in \beta G$.
\end{lemma}

\begin{proof}
Suppose $C$ is left syndetic and fix a finite set $F \subset G$ with $FC = G$. Then
\[
F \oC_q = (F\oC)_q = \overline{FC}_q = G
\]
for all $q \in \beta G$. In particular, $\oC_q$ is non-empty for all $q \in \beta G$. 

Suppose $C$ is \emph{not} left syndetic and note that the set 
\[
Z := \beta G \setminus G \oC
\]
is a non-empty, closed and $G$-invariant set. Indeed, if it was empty, then by compactness of $\beta G$
we could find a finite set $F$ in $G$ such that $F \oC = \beta G$, or equivalently, $FC = G$. However,
this means that $C$ is left syndetic, which contradicts our assumption. Now, for every $q \in Z$, we see
that $\oC_q$ is empty, which finishes the proof. 
\end{proof}

This lemma, combined with Lemma \ref{thick lemma}, now gives:

\begin{proposition}
\label{thick and syndetic}
A set $C \subset G$ is left syndetic if and only if its complement $C^c$ is \emph{not} right thick. 
\end{proposition}

\subsubsection{Density characterizations of left syndetic sets}

We can also establish the following analogue of Proposition \ref{thick help}.

\begin{proposition}
\label{syndetic help}
Suppose $\cC \subset \cM(G)$ is left saturated. Then $C \subset G$ is left syndetic if and only if the lower $\cC$-density of $C$ is positive, i.e. $d_*^{\cC}(C) > 0$.
\end{proposition}

\begin{proof}
Let $C$ be a left syndetic set and choose a finite set $F$ in $G$ such that $FC = G$. Suppose 
$d_*^{\cC}(C) = 0$, or equivalently, there exists $\lambda \in \cC$ with $\lambda(C) = 0$. Since
\[
1 = \lambda(FC) \leq \sum_{f \in F} \lambda(f \cdot C),
\]
we conclude that $\lambda(f \cdot C) > 0$ for at least one $f \in F$. Since $\cC$ is left saturated
and thus in particular left non-singular, we conclude that $\lambda(C) > 0$ as well, which contradicts
our assumption, and hence $d_*^{\cC}(C) > 0$.

Now suppose that $\lambda(C) > 0$ for all $\lambda \in \cC$. We claim that in this case, $\oC_q$
is non-empty for all $q \in \beta G$. By Lemma \ref{syndetic lemma} this implies that $C$ is left syndetic. 
Note that $\oC_q$ is empty if and only if $q$ belongs to the set
\[
Z := \beta G \setminus G \oC.
\] 
Since $\cC$ is left saturated, the set $Z$, being both closed and $G$-invariant, is non-empty if and only
there exists $\lambda \in \cC$ with $\olambda(Z) = 1$, or equivalently (since $\olambda$ is 
non-singular), $\lambda(\oC) = 0$. This finishes the proof. 
\end{proof}

\section{Ergodic theory of measured groups}
\label{ergodic theory}

The aim of this section is to establish some basic properties of measured groups which will be used
in the proofs of Proposition \ref{ergodicity} and Proposition \ref{sum}. 

\subsection{Affine $\mu$-spaces}
Let $G$ be a countable group and assume that $X$ is a Banach space equipped with a 
homomorphism
\[
\rho : G \ra \Isom(X),
\]
where $\Isom(X)$ denotes the group of all \emph{linear} isometries of $X$. Given a probability
measure $\mu$ on $G$, we can define a bounded operator $\rho(\mu)$ on $X$ by the formula
\[
\big\langle \lambda, \rho(\mu)x \big\rangle = \int_G \big\langle \lambda, \rho(g) x \big\rangle \, d\mu(g),
\]
for all $x \in X$ and $\lambda \in X^*$. If $K \subset X^*$ is a weak*-compact and convex subset 
such that 
\[
\rho(\mu)^*K \subset K,
\]
then we shall say that $K$ is an \emph{affine $\mu$-space}. Note that we do \emph{not} assume that the representation $\rho^*$ itself preserves the set $K$. \\

We include a proof of the following standard lemma (sometimes referred to as the Kakutani-Markov Lemma) for completeness. 

\begin{lemma}
\label{fix}
Every affine $\mu$-space admits a fixed point. 
\end{lemma}

\begin{proof}
Suppose $K \subset X^*$ is an affine $\mu$-space, where $X$ is a Banach space, equipped with a representation $\rho : G \ra \Isom(X)$. Fix $\lambda \in K$. Since $\rho(\mu)^*K \subset K$, the 
sequence
\[
\lambda_n := \frac{1}{n} \sum_{k=1}^n \rho(\mu)^{*k}\lambda 
\]
belongs to $K$, and thus admits a weak*-cluster point $\lambda_o \in K$ by weak*-compactness. It is readily checked that $\rho(\mu)^*\lambda_o = \lambda_o$, which finishes the proof. 
\end{proof}

The lemma above especially applies to the regular representation on $X = C(Z)$, where $Z$ is a compact Hausdorff space, equipped with an action of $G$ by homeomorphisms. One readily checks that the space of regular Borel probability measures $\cP(Z)$ is preserved by $\rho(\mu)^{*}$, and identified with the states on $C(Z)$, this is a weak*-compact and convex subset. The set of $\rho(\mu)^{*}$-fixed points in $\cP(Z)$ will be denoted by $\cP_\mu(Z)$, and the elements in $\cP_\mu(Z)$ will be referred to as \emph{$\mu$-stationary} (or $\mu$-harmonic) probability measures.

\begin{corollary}
\label{relative}
Let $X$ and $Y$ be compact $G$-spaces and suppose $\pi : X \ra Y$ is a $G$-equivariant continuous surjection. For every $\mu \in \cP(G)$ and $\eta \in \cP_\mu(Y)$, there exists $\nu \in \cP_\mu(X)$ such that $\pi_*\nu = \eta$. 
\end{corollary}

\begin{proof}
Fix $\mu$ and $\eta$ and define
\[
\cC := \Big\{ \nu \in \cP(X) \, : \, \pi_*\nu = \eta \Big\}.
\]
This set is non-empty by standard arguments, and clearly weak*-compact and convex.
Moreover, since $\eta$ is $\mu$-stationary and $\pi$ is $G$-equivariant, this set is also an 
affine $\mu$-space (however, it is not an affine $G$-space!). By Lemma \ref{fix} there exists a 
$\mu$-fixed $\nu$ point in $\cC$, i.e. a $\mu$-stationary measure on $X$ which projects onto $\eta$
under $\pi$.
\end{proof}

\subsubsection{Properties of $\mu$-stationary measures}

Recall that if $Z$ is a $G$-space, i.e. a compact Hausdorff space equipped with an action of $G$ by 
homeomorphisms, then a Borel probability measure $\nu$ on $Z$ is said to be \emph{non-singular}
(or quasi-invariant) if $\nu(g \cdot N) = 0$ for all $g \in G$, whenever $N \subset Z$ is a $\nu$-null set
of $Z$. We now have the following proposition. 

\begin{proposition}
\label{existence}
For every measured group $(G,\mu)$ and compact $G$-space $Z$, the set $\cP_\mu(Z)$ of 
$\rho(\mu)^*$-fixed probability measures is non-empty. Furthermore, every $\nu \in \cP_\mu(Z)$ 
is non-singular. 
\end{proposition}

\begin{proof}
The first assertion is contained in Corollary \ref{relative} with $Y$ equal to a one-point space, so we shall only prove that every 
element in $\cP_\mu(Y)$ is non-singular. 

Fix $\nu \in \cP_\mu(Z)$ and suppose $B \subset Z$ is a Borel set with $\nu(B) = 0$. 
Since 
\[
\rho(\mu)^{*n}\nu(B) = \int_{G} g_*\nu(B) \, d\mu^{*n}(g) = \nu(B)
\]
we see that $\nu(g^{-1}B) = 0$ for all $g \in \supp(\mu)^n$ and $n \geq 1$ ($G$ is countable). 
Since the support $\supp(\mu)$ is assumed to generate $G$ as a semigroup (by admissibility), 
we conclude that $\nu(g^{-1} \cdot B) = 0$ for all $g \in G$, and thus $\nu$ is non-singular. 
\end{proof}

\subsubsection{Basic properties of the set $\cL_\mu$}

Recall that a set $\cC \subset \cM(G)$ is \emph{left saturated} if every (non-empty) closed left $G$-invariant subset of $\beta G$ supports at least one \emph{non-singular} probability measure of the form $\olambda$ for some $\lambda \in \cC$. Also recall that $\cL_\mu$
denotes the set of left $\mu$-stationary means on $G$.

\begin{corollary}
\label{left saturated}
For every measured group $(G,\mu)$, the set $\cL_\mu$ is left saturated.
\end{corollary}

\begin{proof}
Apply the proposition above to non-empty closed $G$-invariant subsets of the $G$-space $Z = \beta G$.
\end{proof}

In view of Proposition \ref{thick help} and Proposition \ref{syndetic help}, we have also proved the 
following corollary.

\begin{corollary}
\label{mu thick}
For every measured group $(G,\mu)$, a set $T \subset G$ is right thick if and only if it is $\cL_\mu$-conull
and a set $C \subset G$ is left syndetic if and only if the lower $\cL_\mu$-density is positive. 
\end{corollary} 

\subsection{Basic properties of $\mu$-harmonic functions on $G$-spaces}

Let $Y$ be a compact $G$-space, and fix an \emph{admissible} (symmetric) probability measure $\mu$ on $G$, i.e. a probability measure such that $\mu(g) = \mu(g^{-1})$ for all $g$ in $G$ and whose support generates 
$G$ as a semigroup. We shall say that a Borel function $\varphi : Y \ra \bR$ is \emph{$\mu$-harmonic} if
\[
\mu * \varphi(y) := \int_{Y} \varphi(g \cdot y) \, d\mu(g) = \varphi(y), \quad \forall \, y \in Y.
\] 
The following lemma shows that $\mu$-harmonic semicontinuous functions exhibit substantial 
$G$-invariance. This will be a crucial ingredient in the proof of Proposition \ref{sum}.

\begin{lemma}
\label{min invariant}
Suppose $\varphi$ is a lower semicontinuos and $\mu$-harmonic function on $Y$. Then the set of minima 
for $\varphi$ is non-empty, closed and $G$-invariant.
\end{lemma}

\begin{proof}
Let $Z$ denote the set of minima for $\varphi$. Since $\varphi$ is lower semicontinuous, $Z$ is non-empty
and closed, so it suffices to show that $Z$ is $G$-invariant. Fix $y \in Z$ and suppose $g \cdot y \notin Z$. 
Since the support of $\mu$ is assumed to generate $G$ as a semigroup, there exists $n \geq 1$ such that $g$ belongs to 
\[
\supp(\mu^{*n}) = \supp(\mu)^n
\] 
and thus
\[
\int_G \varphi(h \cdot y) \, d\mu^{*n}(h) > \varphi(y), 
\]
which contradicts our assumption that $\varphi$ is $\mu$-harmonic. Hence $g \cdot y \in Z$ and $Z$ is 
$G$-invariant. 
\end{proof}

For $\mu$-harmonic and \emph{continuous} functions, the lemma above can be sharpened as follows.

\begin{lemma}
\label{harmonic invariant}
Let $\nu \in \cP_\mu(Y)$ and suppose $\varphi$ is a continuous and $\mu$-harmonic function on 
the $G$-space $Y$. Then the restriction of $\varphi$ to the support of $\nu$ is $G$-invariant.
\end{lemma}

\begin{proof}
It suffices to prove that for all $\alpha \in \bR$, the intersection of the support of $\nu$ and the set
\[
Z_\alpha := \Big\{ y \in Y \, : \, \varphi(y) = \alpha \Big\} \subset Y
\]
is $G$-invariant. For fixed $\alpha$, we define the function 
\[
\varphi_\alpha(y) := \max(\alpha, \varphi(y)), \quad y \in Y.
\]
Note that $\mu * \varphi_\alpha \geq \varphi_\alpha$ and that the set $Z_\alpha$ coincides with the set of 
minima for $\varphi_\alpha$ for all $\alpha$. If we can show that the restriction of $\varphi_\alpha$ to 
the support of $\nu$ is  $\mu$-harmonic, then Lemma \ref{min invariant} implies that the intersection of 
$Z_\alpha$ with the support of $\nu$ is $G$-invariant. 

However, since $\mu * \varphi_\alpha - \varphi_\alpha \geq 0$ for all $\alpha$, and
\[
\int_Y \big( \mu * \varphi_\alpha - \varphi_\alpha \Big) \, d\nu = 0, \quad \forall \, \alpha \in \bR,
\]
by $\mu$-stationarity of $\nu$, we conclude that $\mu * \varphi_\alpha = \varphi_\alpha$ on the 
support of $\nu$.
\end{proof}

\begin{corollary}
\label{cor invariant}
Let $\nu \in \cP_\mu(Y)$ and suppose $\varphi \in L^{\infty}(Y,\nu)$ satisfies $\rho(\mu) \varphi = \varphi$. Then $\varphi$ is essentially $G$-invariant. 
\end{corollary}

\begin{proof}
Since $L^{\infty}(Y,\nu)$ is a commutative unital C*-algebra, 
the Geland-Naimark Theorem asserts that it is isometrically 
isomorphic to the space of continuous functions on its
(compact) Gelfand spectrum $Z$. Furthermore, the isometric
$G$-action on $L^{\infty}(Y,\nu)$ translates to an action of 
$G$ by homeomorphisms on $Z$. 

Note that $\nu \in \cP_\mu(Y)$ can be thought of as a 
$\rho(\mu)$-invariant state on $L^\infty(Y,\nu)$, and thus 
as a regular $\mu$-stationary Borel probability measure $\bar{\nu}$ on $Z$. One readily checks that $\bar{\nu}$ has full support. 

If $\rho(\mu) \varphi = \varphi$ for some $\varphi \in L^\infty(Y,\nu)$, then the image $\overline{\varphi}$ of $\varphi$ under the Gelfand map is a continuous $\mu$-harmonic function on 
$Z$, so by Lemma \ref{harmonic invariant}, it must be $G$-invariant on $Z$ (since $\bar{\nu}$ has full support). This clearly implies that $\varphi$ is $G$-invariant, which finishes the proof. 
\end{proof}

\subsection{An ergodic theorem for $(G,\mu)$-spaces}
\label{ss ergodic theorem}

Recall that if $Y$ is a $G$-space, then a non-singular Borel probability 
measure $\nu$ on $Y$ is said to be \emph{ergodic} if a $G$-invariant Borel set
either has $\nu$-measure zero or one. In particular, if $\varphi$ is a $G$-invariant Borel function on $Y$, then
\[
\varphi = \int_{Y} \, \varphi \, d\nu, \quad \textrm{a.e. w.r.t. $\nu$}.
\]
Fix a probability measure $\mu$ on $G$ and a $\mu$-stationary Borel probability measure $\nu$ on $Y$. Consider the operator $\rho(\mu)$ on $C(Y)$ defined by
\[
\rho(\mu) \varphi(y) = \int_{G} \varphi(g \cdot y) \, d\mu(g), \quad y \in Y.
\]
Note that since $\nu$ is $\mu$-stationary, $\rho(\mu)$ extends to a bounded operator $\rho_\nu(\mu)$ on $L^1(Y,\nu)$ of norm at most one and thus also defines a bounded operator on $L^2(Y,\nu)$ of norm at most one. However, we stress that the $G$-action itself is far from isometric on $L^1(Y,\nu)$. 

For this operator, we have the following analogue of von Neumann's Ergodic Theorem.

\begin{lemma}
\label{ergodic theorem}
Suppose $\nu \in \cP_\mu(Y)$ is ergodic. Then the $L^2$-limit
\[
\lim_{n \ra \infty} \frac{1}{n} \sum_{k=1}^n \rho_{\nu}(\mu)^k \varphi = \int_{Y} \varphi \, d\nu,
\]
exists for all $\varphi \in L^{2}(Y,\nu)$.
\end{lemma}

\begin{proof}
Note that the operator norm of $\rho_\nu(\mu)$ on $L^2(Y,\nu)$ is bounded by one. Hence, by Lorch's Ergodic Theorem (Theorem 1.2, Chapter 2 in \cite{Kr}), the $L^2$-limit exists for all $\varphi$ and is 
$\rho_\nu(\mu)$-invariant. Since $\nu$ is ergodic, Corollary \ref{cor invariant} implies that the limit is 
essentially constant and thus equal to the integral. 
\end{proof}

By a simple approximation argument, using the fact that $L^2(Y,\nu)$ is a dense subspace of 
$L^1(Y,\nu)$, we can extend the convergence asserted in the previous lemma to $L^1$-convergence
for all $L^1$-functions. Furthermore, one can also prove statements about almost everywhere convergence upon 
referring to classical results. 

\subsubsection{Estimating the upper $\cF_\mu$-density}

Recall that 
\[
\cF_\mu := \Big\{ \lambda \in \cL_\mu \, : \, \lambda |_{\cH^{\infty}(G,\mu)} = \delta_{e} \Big\}.
\]
We have already noticed that every weak*-accumulation point in $\ell^{\infty}(G)^*$ of the sequence 
\[
\lambda_n := \frac{1}{n} \sum_{k=1}^{n} \mu^{*n},
\]
belongs to $\cF_\mu$. In particular, if $Y$ is a compact $G$-space, then
\[
d^*_{\cF_\mu}(B_y) \geq \limsup_{n \ra \infty} \frac{1}{n} \sum_{k=1}^{n} \mu^{*n}(B_y)
\]
for all $y \in Y$, where 
\[
B_y = \Big\{ g \in G \, : \, g \cdot y \in B \Big\}, \quad y \in Y.
\]
The following corollary shows that there is a simple lower bound for the upper $\cF_\mu$-density
(and thus for the upper $\cL_\mu$-density as well) for these kinds of subsets. 

\begin{corollary}
\label{positive density}
Suppose $\nu \in \cP_\mu(Y)$ is ergodic. Then, for every Borel set $B \subset Y$, there exists 
a $\nu$-conull Borel set $Y' \subset Y$ such that
\[
d^*_{\cF_\mu}(B_y) \geq \nu(B), \quad \forall \, y \in Y'.
\]
\end{corollary}

\begin{proof}
It suffices to prove that there exists a conull subset $Y' \subset Y$ such that
\[
\limsup_{n \ra \infty} \frac{1}{n} \sum_{k=1}^n \mu^{*k}(B_y) > 0
\]
for all $y \in Y'$. Since 
\[
\sum_{k=1}^n \mu^{*k}(B_y) = \sum_{k=1}^n \rho(\mu)^{*k} \chi_{B}(y),
\]
where $\rho(\mu)$ denotes the bounded operator on $L^2(Y,\nu)$ introduced above,
we have by Lemma \ref{ergodic theorem}, 
\[
\lim_n \frac{1}{n} \sum_{k=1}^n \rho(\mu)^{*k} \chi_{B} = \nu(B) > 0
\]
in the \emph{$L^2$-norm}. Hence we can extract a subsequence such that the limit
exists almost everywhere along this subsequence.
\end{proof}

Note that since $\mu$ is assumed to be \emph{symmetric}, the proof of the previous
corollary also implies that sets of the form $B_y^{-1}$ are $\cF_\mu$-large, for all 
$y \in Y'$.  This has an important consequence as the following corollary shows.

\begin{corollary}
Suppose $A \subset G$ is $\cL_\mu$-large. Then there exists a $\cF_\mu$-large set
$B \subset G$ such that $B^{-1}$ is $\cF_\mu$-large as well and 
\[
AA^{-1} \supset BB^{-1}.
\]
\end{corollary}

\begin{proof}
Let $\lambda$ be an element in $\cL_\mu$ such that $\olambda$ is an \emph{ergodic}
$\mu$-stationary probability measure on $\beta G$ and the set $\oA \subset \beta G$ has positive $\olambda$-measure (such elements exist 
by Proposition \ref{ext erg} below). The previous corollary guarantees the existence of $q \in \beta G$ such that both 
sets $\oA_q$ and $\oA_q^{-1}$ are $\cF_\mu$-large. \\

Fix an increasing exhaustion $(F_n)$ of $G$ by finite sets. For every $n$, 
there exists by Lemma \ref{UF2} an element $g_n \in G$ such that
\[
\oA_q \cap F_n = (A \cap F_n \cdot g_n) \cdot g_n^{-1}
\]
In particular, 
\[
AA^{-1} \supset (A \cap F_n \cdot g_n)(A \cap F_n \cdot g_n)^{-1} = (\oA_q \cap F_n)(\oA_q \cap F_n)^{-1}
\]
for all $n$, and thus
\[
AA^{-1} \supset \oA_q \oA_q^{-1}. 
\]
Hence, if we set $B := \oA_q$, then this set satisfies the desired properties. 
\end{proof}

\subsection{Poisson boundaries of a measured groups}

Recall that a \emph{bounded} function $f$ on $G$ is \emph{left $\mu$-harmonic} (note that we shall always assume
that $\mu$ is \emph{symmetric}) if the equation
\[
\mu * f(g) := \int_G f(hg) \, d\mu(h) = f(g)
\] 
holds for all $g \in G$. Clearly, a constant function is always left $\mu$-harmonic, and there are instances
when constants are the only (bounded) $\mu$-harmonic functions (e.g. when $G$ is a countable abelian group and $\mu$ is any admissible probability measure on $G$). 

We shall denote by $\cH^{\infty}_l(G,\mu)$ the space of all bounded left $\mu$-harmonic 
functions on $G$. Easy examples show that the pointwise product of two left $\mu$-harmonic functions
may fail to be left $\mu$-harmonic, so $\cH^{\infty}_l(G,\mu)$ is \emph{not} a sub-algebra of 
$\ell^{\infty}(G)$, and thus do not quite fit onto the framework we have set up so far. 

However, there is another product (which is a special case of a more general construction in operator
system theory called the \emph{Choi-Effros} product), which does turn $\cH^{\infty}_l(G,\mu)$ into an
algebra. In our setting, this product, here denoted by $\square$, can be defined by
\[
(\varphi \, \square \, \psi)(g) := \lim_{n \ra \infty} \int_{G} \varphi(hg) \, \psi(hg) \, d\mu^{*n}(h), \quad g \in G,
\]
for $\varphi, \psi \in \cH^{\infty}(G,\mu)$. One can show that the limit exists by a simple use of the Martingale
Convergence Theorem. We refer the reader to \cite{Bab} for details about this product; especially the fact that 
$(\cH^{\infty}_l(G,\mu), \square)$ is a commutative C*-algebra. 

Another serious deviation from the set up so far is that $\cH^{\infty}_l(G,\mu)$ is invariant under the right
regular representation $r$ on $\ell^{\infty}(G)$, but \emph{not} necessarily under the left regular representation. This right action on $\cH_l^{\infty}(G,\mu)$ gives rise to a (right) action of $G$ by homeomorphisms on the Gelfand spectrum 
\[
B := \Delta(\cH_l^{\infty}(G,\mu)),
\]
which we shall consistently write as a left action, i.e. $g \cdot b := bg^{-1}$, for $b \in B$ and $g \in G$. 
 
Note that the evaluation map $\delta_e$ on $\cH^{\infty}_l(G,\mu)$ is a $r(\mu)$-invariant state, and thus
corresponds to a $\mu$-stationary probability measure $\nu_o$ on $B$. Indeed, 
\[
\big\langle r(\mu)^*\delta_e, f \big\rangle = \int_{G} f(h) \, d\mu(h) = \big\langle \delta_e , f \big\rangle,
\quad \forall \, f \in \cH^{\infty}(G,\mu).
\]
If $f \in C(B)$, we define $\hat{f} \in \ell^{\infty}(G)$ by
\[
\hat{f}(g) := \int_B f(g^{-1} \cdot b) \, d\nu_o(b), \quad g \in G.
\]
Since $\mu$ is symmetric, $\hat{f}$ belongs to $\cH_l^{\infty}(G,\mu)$ as the following calculation shows. 
\begin{eqnarray*}
\int_{G} \hat{f}(hg) \, d\mu(h) 
&=& 
\int_{G} \int_B f(g^{-1} \cdot h^{-1} \cdot b) \, d\nu_o(b) \, d\mu(h) \\
&=&
\int_B f(g^{-1} \cdot b) \, d(\mu * \nu_o)(b)  = \hat{f}(g), \quad \forall \, g \in G.
\end{eqnarray*}
Note that we can also write $\hat{f}$ as
\[
\hat{f}(g) = \big\langle \nu, r(g)f \big\rangle_{C(B)} 
= 
\big\langle r(g)^*\delta_e, f' \big\rangle_{\cH^{\infty}_l(G,\mu)} = f'(g), \quad \forall \, g \in G,
\]
where $f'$ denotes the image of $f \in C(B)$ in $\cH_l^{\infty}(G,\mu)$ under the Gelfand map (for
the Choi-Effros product), so we conclude that $\hat{f}$ completely determines $f$. This observation
suggests the following proposition, whose proof can be found in \cite{Bab}.

\begin{proposition}
\label{poissonx}
The $\mu$-stationary probability measure $\nu_o$ on $B$ is ergodic and the map 
\[
f \mapsto \hat{f}
\]
is an isometric $*$-isomorphism between $L^{\infty}(B,\nu_o)$ and $\cH^{\infty}(G,\mu)$.
\end{proposition}
\begin{remark}
The Gelfand spectrum of $\cH_l^{\infty}(G,\mu)$ with the Choi-Effros product, together with the measure 
$\nu_o$, is sometimes referred to as
the \emph{Poisson boundary of $(G,\mu)$}. However, since this measure space is only standard when
$\cH^{\infty}_l(G,\mu)$ reduces to the constants, one often prefers to choose other (countably generated) Borel representations which will yield the same $L^\infty$-space, and then such a space is also referred to as the Poisson boundary of $(G,\mu)$. In this paper, we shall only work with the above realization of the Poisson boundary. 
\end{remark}

\subsection{Left ergodicity of $\cF_\mu \subset \cM(G)$}
\label{poisson}
As we have seen, the space $\cP_\mu(Y)$ of $\mu$-stationary Borel measures on a $G$-space $Y$,
is always a non-empty, weak*-compact and convex set. In particular, by Krein-Milman's Theorem, it has extremal points. We shall denote the 
set of extremal points by $\Ext(\cP_\mu(Y))$. \\

Let $\Erg_\mu(Y)$ denote the set of $\mu$-stationary \emph{ergodic} probability measures on $Y$.
In the weak*-topology, this is a nice $G_\delta$-set \emph{if} $Y$ is second countable. When $Y$ is
not second countable, then there is no reason to expect this set to be even Borel measurable. Particularly, this is the case for the Poisson boundary $B$ defined in the previous subsection. In any case,
we do have the following equality of sets.  

\begin{proposition}
\label{ext erg}
For every measured group $(G,\mu)$ and compact $G$-space $Y$, we have 
\[
\Ext(\cP_\mu(Y)) = \Erg_\mu(Y).
\] 
In particular, $\Erg_\mu(Y)$ is non-empty. 
\end{proposition}

\begin{proof}
First note that if $\nu \in \cP_\mu(Y)$ and $X \subset Y$ is a $G$-invariant $\nu$-measurable set with 
measure 
\[
0 < \nu(X) < 1,
\]
then the re-normalized restriction of $\nu$ to $X$ clearly belongs to $\cP_\mu(Y)$. Hence, if $\nu$ is 
\emph{not} ergodic, it \emph{cannot} be extremal in $\cP_\mu(Y)$. 

The converse is slightly more difficult. Suppose $\nu$ is not extremal, but still ergodic, so that we can write 
\[
\nu = \alpha_1 \cdot \nu_1 + \alpha_2 \cdot \nu_2,
\]
for some positive $\alpha_1, \alpha_2$ with $\alpha_1 + \alpha_2 = 1$, where $\nu_1, \nu_2 \in \cP_\mu(Y)$. 

Note that both $\nu_1$ and $\nu_2$ must be absolutely continuous with respect to $\nu$, so 
by Radon-Nikodym's Theorem, it suffices to prove that if $\rho \in L^1(Y,\nu)$ has the property that the $d\nu' = \rho d\nu$ is a $\mu$-stationary, probability measure, then $\rho$ is equal to 
one $\nu$-almost everywhere. 

Let $r > 0$ and let $\rho_r$ denote the cut-off of $\rho$ for values greater than $r$ so that $\rho_r$ is a bounded $\nu$-measurable function. Note that
\[
\int_Y \varphi \, d\nu' 
=
\int_Y A_n \varphi \, d\nu' 
\geq
\int_{Y} \big( A_n \varphi \big) \, \rho_r \, d\nu
\]
for all $\varphi \in L^\infty(Y,\nu)_{+}$, where 
\[
A_n \varphi := \frac{1}{n} \sum_{k=1}^n \rho(\mu)^{*k} \varphi. 
\]
Here, $\rho(\mu)$ is the bounded operator on $L^2(Y,\nu)$ introduced in Lemma \ref{ergodic theorem}.
Since $\nu$ is ergodic, this lemma applies, and thus the sequence $A_n \varphi$ converges to the 
$\nu$-integral of $\varphi$ in the $L^2$-norm. Hence, for all $\varphi \in L^\infty(Y, \nu)_{+}$, we have
\[
\int_Y \varphi \, d\nu' \geq \lim_n \big\langle \rho_r, A_n \varphi \big\rangle_{L^2(Y,\nu)} 
=  
\Big(\int_{Y} \rho_r \, d\nu \Big) \cdot \int_Y \varphi \, \, d\nu.
\]
Upon letting $r$ tend to infinity, we see that 
\[
\int_Y \varphi \, d\nu' \geq \int_Y \varphi \, \, d\nu, \quad \forall \, \varphi \in L^\infty(Y,\nu)_+
\]
which readily implies that $\rho \geq 1$ $\nu$-almost everywhere. However, since the $\nu$-integral
equals one, we conclude that $\rho$ must in fact be $\nu$-almost everywhere equal to one, which finishes the proof. 
\end{proof}

Recall that a subset $\cC \subset \cL_\mu$ is \emph{left ergodic} if $\tau^*\lambda$ 
is an \emph{ergodic} $\mu$-stationary Borel probability measure on $\beta G$ whenever $\lambda$ 
is an \emph{extremal} point in $\cC$, where $\tau : C(\beta G) \ra \ell^{\infty}(G)$ is the $G$-equivariant
map defined by
\[
\tau(\varphi)(g) := \varphi(g \cdot \bar{e}), \quad g \in G.
\] 
The set of \emph{Furstenberg-Poisson means} is defined by
\[
\cF_\mu := \Big\{ \lambda \in \cL_\mu \, : \, \lambda |_{\cH^{\infty}_l(G,\mu)} = \delta_e \Big\}.
\]
The following corollary isolates one of the most important properties of $\cF_\mu$ which will be 
used in the proof of Theorem \ref{thm1}.

\begin{corollary}
\label{left ergodic}
For every measured group $(G,\mu)$, the set $\cF_\mu \subset \cM(G)$ is left ergodic. 
\end{corollary}

\begin{proof}
Let $\lambda \in \cF_\mu$ and suppose that the $\mu$-stationary probability measure
$\olambda$ on $\beta G$ is \emph{not} ergodic. By Lemma \ref{ext erg}, $\olambda$ 
is not extremal in the set $\cP_\mu(\beta G)$, so we can write
\[
\olambda = \alpha \cdot \olambda_1 + \alpha_2 \cdot \olambda_2,
\]
for some $\olambda_1, \olambda_2 \in \cP_\mu(Y)$ and \emph{positive} $\alpha_1, \alpha_2$
with $\alpha_1 + \alpha_2 = 1$. We claim that the corresponding means $\lambda_1$ and $\lambda_2$ 
must belong to $\cF_\mu$, and thus $\lambda$ is \emph{not extremal $\cF_\mu$}, which finishes the proof. \\

To prove this claim, we first note that restriction of $\lambda$ to $\cH_l^{\infty}(G,\mu)$ is 
$\mu$-stationary  \emph{on the right}, i.e.
\[
\lambda(r(\mu)f) = \lambda(f), \quad \forall \, f \in \cH^{\infty}_l(G,\mu),
\]
where $r$ denotes the right regular representation on $\cH^{\infty}_l(G,\mu)$. Note that if $f$ is left $\mu$-harmonic, then so is $r(\mu)^{k}f$ for all $k$ and thus
\[
f(e) = \alpha_1 \cdot \lambda_1(r(\mu)^k f) + \alpha_2 \cdot \lambda_2(r(\mu)^k f)
\]
for all left $\mu$-harmonic functions $f$ and for all $k$. In particular, 
upon averaging over $k$, we may assume that the \emph{restrictions} to $\cH^{\infty}_{l}(G,\mu)$ of both $\lambda_1$ and $\lambda_2$ are fixed by $r(\mu)$. \\

Recall that the map $f \mapsto \hat{f}$ from $L^{\infty}(B,\nu_o)$ to $\cH^{\infty}_l(G,\mu)$ defined in Proposition \ref{poissonx} is an isometry, so the functionals 
\[
\nu_1(f) := \lambda_1(\hat{f}) \qand \nu_2(f) := \lambda_2(\hat{f}) 
\]
can be viewed as regular Borel probability measures on the Poisson boundary $B$ of $(G,\mu)$.
Furthermore, 
\[
\lambda(\hat{f}) = \hat{f}(e) = \nu_o(f), \quad \forall \, f \in C(B).
\]
We wish to prove that 
\[
\nu_o = \nu_1 = \nu_2.
\]
Note that since 
\[
\nu_o = \alpha_1 \cdot \nu_1 + \alpha_2 \cdot \nu_2
\]
and $\nu_o$ is ergodic by Proposition \ref{poissonx}, it suffices to show that $\nu_1$ and 
$\nu_2$ are $\mu$-stationary. Indeed, by Proposition \ref{ext erg}, $\nu_o$ is extremal in $\cP_\mu(B)$, so such an equation must force all measures to be equal. 

Since the restrictions of $\lambda_1$ and $\lambda_2$ to $\cH_l^{\infty}(G,\mu)$ satisfy 
\[
r(\mu)^*\lambda_i = \lambda_i, \quad i = 1, 2,
\]
we have
\begin{eqnarray*}
\mu * \nu_i(f) 
&=& 
\int_G \int_B f(g \cdot b) \, d\nu_i(b) \, d\mu(g)\\
&=& 
\int_G \int_G \Big( \int_B f(g \cdot h^{-1} \cdot b) \, d\nu_o(b) \Big) \, d\lambda_i(h) \, d\mu(g) \\
&=&
\int_G \int_G \Big( \int_B f((h \cdot g^{-1})^{-1} \cdot b) \, d\nu_o(b) \Big) \, d\lambda_i(h) \, d\mu(g) \\
&=&
\int_G \Big( \int_B f(h^{-1} \cdot b) \, d\nu_o(b) \Big) \, d(r(\mu)^*\lambda_i)(h) \\
&=&
r(\mu)^*\lambda_i(\hat{f}) = \lambda_i(\hat{f}) = \nu_i(f),
\end{eqnarray*}
for $i = 1, 2$ and all $f \in C(B)$. Thus $\nu_1$ and $\nu_2$ belong to $\cP_\mu(B)$, which finishes 
the proof. 
\end{proof}

\subsection{Failure of left syndeticity for difference sets}
\label{pitfall}
We shall now construct examples of countable measured groups 
$(G,\mu)$ and $\cF_\mu$-large sets $A \subset G$ with the property 
that their difference sets $AA^{-1}$ are right piecewise left syndetic but \emph{not} left syndetic. 

\subsubsection{Basic setup}
Let $G$ denote the free group on the free symbols $\{a,b\}$. Let 
$\bT = \bR/\bZ$ and fix an irrational number in $\bT$. Also
choose a homeomorphism $T$ on $\bT$ which fixes two distinct 
points $x_{+}$ and $x_{-}$ on $\bT$ and has the property that 
for every open neighborhood $U$ of $x_{+}$ and every closed 
subset $B \subset \bT$ which does not contain $x_{-}$, there 
exists $n$ such that
\[
T^{n}(B) \subset U.
\]
We define an action of $G$ on $\bT$ by
\[
a \cdot x = x + \alpha
\qand
b \cdot x = T(x), \quad x \in \bT.
\]

\begin{lemma}
\label{fail1}
For every non-empty open set $U \subset \bT$ and proper closed subset 
$B \subset \bT$, there exists $g \in G$ such that $g \cdot B \subset U$.
\end{lemma}

\begin{proof}
Since $\alpha$ is irrational, the restricted action by the subgroup
$a^{\bZ}$ is minimal, and thus we can find $m, n \in \bZ$ such 
that  $a^{m} \cdot U$ is an open neighborhood of $x_{+}$ and 
$a^{n} \cdot B$ does not contain $x_{-}$. By the defining property
of $T$, we can now find $k \in \bZ$ such that
\[
b^{k} \cdot (a^{m} \cdot B) \subset a^{n} \cdot U,
\]
or equivalently, $(a^{-n} \, b^{k} \, a^{m}) \cdot B \subset U$. 
\end{proof}

\begin{corollary}
\label{fail2}
For every finite set $F \subset G$ and closed subset $A \subset \bT$ 
with empty interior, there exists $g \in G$ such that $FA \cap g \cdot A$ 
is an empty subset of $\bT$.
\end{corollary}

\begin{proof}
By Lemma \ref{fail1} it suffices to show that $FA$ is a proper (closed is immediate) subset for 
every finite set $F \subset G$. Indeed, if this is the case, then we can choose 
\[
B := A \qand U:= (FA)^c 
\]
and find (using the previous lemma) an element $g \in G$ such that $g \cdot B \subset U$, or equivalently, 
\[
FA \cap g \cdot A = \emptyset.
\]
However, if $FA = \bT$ for some finite set, then by Baire's Category 
Theorem, at least one of $fA$, with $f \in F$, has non-empty interior. Since $G$ acts by homeomorphisms, this would imply that $A$ has non-empty interior, which we have assumed it does not have. 
\end{proof}

Fix an admissible (symmetric) $\mu \in \cP(G)$ and choose an 
ergodic $\mu$-stationary Borel probability measure $\nu$ on 
$\bT$. The support of $\nu$ is a non-empty closed $G$-invariant
subset, so in particular it is invariant under the \emph{dense} 
subgroup $\bZ \cdot \alpha \subset \bT$, and must thus be equal
to $\bT$.

Recall that if $A$ is any subset of $\bT$ and $x$ is a 
point in $\bT$, then 
\[
A_{x} := \Big\{ g \in G \, : \, g \cdot x \in A \Big\} \subset G. 
\]
The following lemma is now an immediate consequence of Corollary \ref{positive density}. 

\begin{lemma}
\label{posmeasure}
For every Borel set $A \subset \bT$ with positive $\nu$-measure, 
there exists a $\nu$-conull Borel set $X \subset \bT$ such that 
$A_x \subset G$ is $\cF_\mu$-large for all $x \in X$. 
\end{lemma}

Before we can construct our examples, we need the following general
result.

\begin{lemma}
Let $X$ be a compact separable Hausdorff space and let $\nu$ be a non-atomic regular Borel probability measure on $X$ with full support. Then
there exists a closed set $B \subset X$ with empty interior and positive
$\nu$-measure. 
\end{lemma}

\begin{proof}
Since $\nu$ is a non-atomic and regular Borel probability measure, every point $x$ in $X$ admits a decreasing sequence $(U_n)$ of open neighborhoods with the property that
\[
\forall \, n,  \quad\nu(U_n) \leq \frac{1}{n}.
\] 
Let $Y$ be a countable dense subset of $X$ and fix an enumeration 
$(y_n)$ of the elements in the set $Y$. Let $(q_n)$ be a sequence of positive
real numbers with
\[
\sum_{n} q_n < 1.
\]
Since $\nu$ has full support, we can, for every $n$, choose a neighborhood $U_n$ of $y_n$ with $\nu$-measure at most $q_n$. 
If we let
\[
U = \bigcup_{n} U_n,
\] 
then, since $Y$ is dense, $U$ is open and dense, and thus $B = U^c$ is closed with no interior and
\[
\nu(B) = 1 - \mu(B) \geq 1 - \sum_{n} q_n > 0,
\]
which finishes the proof.
\end{proof}

\begin{remark}
The assumption that $X$ is separable is crucial. Indeed, let $(Z,\eta)$ be
any non-atomic probability measure space and let $X$ denote the Gelfand
spectrum of $L^{\infty}(Z,\eta)$. Then $X$ is a compact (non-separable) Hausdorff space and the corresponding probability measure $\bar{\eta}$ on $X$ (viewing $\eta$ as a positive linear functional on 
$L^{\infty}(Z,\eta)$) is a (non-atomic) \emph{normal} measure 
(with full support), i.e. for every Borel set $B \subset X$, we have
\[
\bar{\eta}(B^{o}) = \bar{\eta}(B).
\]
In particular, a Borel subset of $X$ has positive $\nu$-measure if and only if $B$ has non-empty interior. See Chapter 1 in \cite{Gam} for references.
\end{remark}

Since $\nu$ has full support and $\bT$ is separable, the previous lemma asserts that there is a proper closed subset $A \subset \bT$ with positive $\nu$-measure, but empty interior. By Lemma \ref{posmeasure} there 
exists a $\nu$-conull Borel set $X \subset \bT$ such that the sets 
$A_{x} \subset G$ are $\cF_\mu$-large for all $x \in X$. 

We claim that the difference set $A_{x}A_{x}^{-1}$ is \emph{not} left syndetic. Indeed, assume that $FA_x A_x^{-1} = G$ for some $x \in X$. Then, for all $g \in G$, we have
\[
FA_x \cap g \cdot A_x = (FA \cap g \cdot A)_x \neq \emptyset. 
\]
In particular, the set $FA \cap g \cdot A \subset S^1$ is non-empty for all $g \in G$. However, this contradicts Corollary \ref{fail2}.

\section{Proof of Proposition \ref{ergodicity}}
\label{Proof of ergodicity}

Let $(G,\mu)$ be a countable measured group. A set $\cC \subset \cL_\mu$ is \emph{left ergodic}
if every extremal element in $\cC$ corresponds to an \emph{ergodic} $\mu$-stationary regular Borel probability measure on the Stone-\v{C}ech compactification $\beta G$ (with respect to the induced left action) via the Gelfand map
\[
\tau(\varphi)(g) = \varphi(g \cdot \bar{e}), \quad g \in G,
\]
By Corollary \ref{left ergodic}, the set 
$\cF_\mu \subset \cL_\mu$ of Furstenberg-Poisson means on $G$ is left ergodic, so Proposition \ref{ergodicity} follows immediately
from the following lemma.

\begin{lemma}
Suppose $\cC \subset \cM(G)$ is left ergodic and $A \subset G$ is $\cC$-large. Then, 
\[
\sup d_{\cC}^*(FA) = 1,
\]
where the supremum is taken over all finite subsets $F \subset G$.
\end{lemma}

\begin{proof}
By Corollary \ref{max2}, there exists an extremal $\lambda \in \cC$ such that 
\[
d^*_{\cC}(A) = \lambda'(A) = \overline{\lambda}(\overline{A}) > 0.
\]
Since $\cC$ is left ergodic, $\overline{\lambda}$ is ergodic for the 
induced left action of $G$ on $\beta G$. In particular, by the $\sigma$-additivity 
of $\overline{\lambda}$, we have
\[
1 = \overline{\lambda}(G\overline{A}) = \sup \overline{\lambda}(F\overline{A}),
\]
where the supremum is taken over all finite subsets $F \subset G$. Since $F \overline{A}$ is 
clopen in $\beta G$ for every finite $F$, we conclude that
\[
1 = \sup \overline{\lambda}(F\overline{A}) = \sup \lambda'(FA) \leq \sup d^*_{\cC}(FA),
\]
which finishes the proof.
\end{proof}

\section{Proof of Proposition \ref{sum}}
\label{proof of sum}

Recall that a set $C \subset G$ is \emph{left syndetic} (or left relatively dense) if there exists a 
finite set $F \subset G$ such that $FC = G$. By Lemma \ref{syndetic lemma}, this is equivalent 
to positivity of the lower $\cL_\mu$-density. 

\begin{proposition}
\label{lower bound}
Suppose $C \subset G$ is left syndetic. Then
\[
d_*^{\cL_\mu}(A^{-1}C) \geq d^*_{\cF_\mu}(A)
\]
for all $A \subset G$.
\end{proposition}

This proposition can be used to give a quick proof of Proposition \ref{sum}.

\begin{proof}[Proof of Proposition \ref{sum}]
Let $(G,\mu)$ be a countable measured group, and assume for contradiction that $A, B \subset G$ satisfy
\[
d^*_{\cF_\mu}(A) + d^*_{\cL_\mu}(B) > 1,
\]
but the product set $AB$ is not right thick. Then, by Proposition \ref{thick and syndetic}, the complementary 
set $C := (AB)^c$ is syndetic,
and since $A^{-1}C \subset B^c$, we have
\[
d_*^{\cL_\mu}(A^{-1}C) \leq d_*^{\cL_\mu}(B^c) = 1 - d^*_{\cL_\mu}(B) < d^*_{\cF_\mu}(A).
\]
However, this contradicts Proposition \ref{lower bound}.
\end{proof}

\subsection{Proof of Proposition \ref{lower bound}}

Recall that if $A$ is a subset of a countable group $G$, we denote by $\oA$ the corresponding 
"closure" set in the Stone-\v{C}ech compactification $\beta G$. If $q \in \beta G$, we define
\[
\oA_q := \Big\{ g \in G \, : \, g \cdot q \in \oA \Big\} \subset G.
\]
In particular, $A = \oA_{\bar{e}}$. 

\begin{lemma}
\label{lsc harmonic}
Let $A \subset G$ and $\eta \in \cL_\mu$. For every Borel set $D \subset \beta G$, the function
\[
\psi(q) := \oeta(\oA^{-1}_q D), \quad q \in \beta G,
\]
is $\mu$-harmonic and equal to the supremum of an increasing sequence of continuous functions on 
$\beta G$.
\end{lemma}

\begin{proof}
Fix an increasing exhaustion $(F_n)$ of $G$ by finite sets, and define
\[
\psi_n(q) := \oeta((\oA_q \cap F_n)^{-1} D), \quad q \in \beta G,
\]
for all $n$. By $\sigma$-additivity, $\psi = \sup_n \psi_n$, so it suffices to show that 
$\psi_n$ is continuous for every $n$. 

Fix $q \in \beta G$ and $n$. By Lemma \ref{UF2}, the set of $p \in \beta G$ such that 
\[
\oA_p \cap F_n = \oA_q \cap F_n
\]
and thus $\psi_n(p) = \psi_n(q)$, is clopen. This clearly implies that $\psi_n$ is continuous. 

To show that $\psi$ is $\mu$-harmonic, we first note that 
\[
\oA_{g \cdot q} = \oA_{q} \cdot g^{-1}, \quad \forall \, g \in G,
\]
and thus,
\[
\psi(g^{-1} \cdot q) = \oeta(g^{-1} \cdot \oA_q^{-1} D) = \int_{\beta G} \chi_{\oA^{-1}_q D}(g \cdot p) \, d\oeta(p).
\]
Since $\mu * \oeta = \oeta$, we have $\mu * \psi = \psi$ (remember that $\mu$ is symmetric), which finishes the proof. 
\end{proof}

\begin{lemma}
\label{lower bound help3}
Suppose $A \subset G$. If $C \subset G$ is left syndetic, then
\[
\sup_{\eta \in \cF_\mu} \int_{\beta G} \oeta(\oC_{q}^{-1} \oA) \, d\olambda(q) \geq d^*_{\cF_\mu}(A),
\]
for all $\lambda \in \cL_\mu$.
\end{lemma}

\begin{proof}
Fix $A, C \subset G$, with $C$ left syndetic, and $\lambda$ and $\eta$ in $\cL_\mu$. By 
Lemma \ref{lsc harmonic}, the function
\[
\psi(q) := \oeta(\oC_{q}^{-1} \oA), \quad q \in \beta G,
\]
is a lower semicontinuous $\mu$-harmonic function on $\beta G$. By Lemma \ref{min invariant}, 
the set $Z$ of minima for $\psi$ is a non-empty, closed and $G$-invariant set, and since $C$ is left
syndetic, $\oC_q$ is non-empty for all $q \in \beta G$ by Lemma \ref{syndetic lemma}. \\

In particular, there exists $q_o \in Z$ such that $e \in \oC_{q_o}$. Hence, 
\[
\oeta(\oA) \leq \oeta(\oC_{q_o}^{-1} \oA) \leq \int_{\beta G} \oeta(\oC_{q}^{-1} \oA) \, d\olambda(q).
\]
Since $\lambda, \eta \in \cL_\mu$ are arbitrary, we are done. 
\end{proof}

The following proposition is the main result of this subsection. Combined 
with Lemma \ref{lower bound help3} above, it immediately implies Proposition \ref{lower bound2}.

\begin{proposition}
\label{lower bound2}
Suppose $A, C \subset G$. Then there exists an extremal $\lambda \in \cL_\mu$ such that
\[
d_*^{\cL_\mu}(A^{-1}C) \geq \sup_{\eta \in \cF_\mu} \, \int_{\beta G} \oeta(\oC_q^{-1} \oA) \, d\olambda(q).
\]
\end{proposition}

\subsection{Proof of Proposition \ref{lower bound2}}

\begin{lemma}
\label{lower bound help4}
Suppose $\psi$ is a bounded $\mu$-harmonic function on $\beta G$, equal to the supremum of an 
increasing sequence of continuous functions. Then,
\[
\psi(\bar{e}) \geq \sup_{\eta \in \cF_\mu} \, \int_{\beta G} \psi \, d\oeta.
\]
\end{lemma}

\begin{proof}
Note that the function
\[
\varphi(g) := \psi(g \cdot \bar{e}), \quad g \in G,
\]
belongs to $\cH^{\infty}_l(G,\mu)$, and thus
\[
\eta(\varphi) = \varphi(e) = \psi(\bar{e}),
\]
for all $\eta \in \cF_\mu$. By Lemma \ref{lower bound help1}, 
\[
\eta(\varphi) \geq \int_{\beta G} \psi(q) \, d\oeta(q),
\]
for all $\eta \in \cF_\mu$, which finishes the proof. 
\end{proof}

\begin{proof}[Proof of Proposition \ref{lower bound2}]
Fix $A, C \subset G$. By Corollary \ref{max2}, there exists an extremal $\lambda \in \cL_\mu$
such that
\[
d_*^{\cL_\mu}(A^{-1}C) = \lambda'(A^{-1}C) = \lambda'((A^{-1} \oC)_{\bar{e}}).
\]
Since $A^{-1}\oC$ is a countable union of clopen sets in $\beta G$, Lemma \ref{lower bound help2}
applies and
\[
\lambda'((A^{-1} \oC)_{\bar{e}}) \geq \olambda(A^{-1} \oC).
\]
Define the function
\[
\psi(q) := \olambda(\oA_q^{-1} \oC), \quad q \in \beta G.
\]
By Lemma \ref{lsc harmonic}, $\psi$ is $\mu$-harmonic and equal to the supremum of an increasing
sequence of continuous functions. Hence, Lemma \ref{lower bound help4} applies, and
\[
\psi(e) = \olambda(\oA^{-1}_{\bar{e}} \oC) \geq \int_{\beta G} \olambda(\oA^{-1}_q \oC) \, d\oeta(q),
\]
for all $\eta \in \cF_\mu$. By Fubini's Theorem, we have
\[
\int_{\beta G} \olambda(\oA^{-1}_q \oC) \, d\oeta(q) 
=
\int_{\beta G} \int_{\beta G} \chi_{G(\oA \times \oC)} \, 
d\olambda \otimes \oeta 
=
\int_{\beta G} \oeta(\oC^{-1}_q \oA)  \, d\olambda(q),
\]
which finishes the proof.
\end{proof}

\section{Proof of Proposition \ref{bohr appears}}
\label{Bohr}

The proof of Proposition \ref{bohr appears} will be done in several steps. Some of these steps are 
quite standard and well-known to experts, at least for amenable groups, but we still give the proofs
for completeness. 
 
\subsection{Reductions}

We shall try to re-cast Proposition \ref{bohr appears} in a more C*-algebraic language. This will 
simplify, and hopefully conceptualize, the approach. 

Recall that if $\cB$ is a left $G$-invariant unital sub-C*-algebra of $\ell^{\infty}(G)$, then the 
Gelfand space $\Delta(\cB)$ of $\cB$ is a compact Hausdorff space, equipped with an action
of $G$ by homeomorphisms, and admitting a $G$-transitive point $x_o$. Furthermore, there 
is a left $G$-equivariant $*$-isomorphism
\[
\tau_{\cB} : C(\Delta(\cB)) \ra \cB
\]
given by
\[
\tau_{\cB}(f)(g) := f(g \cdot x_o), \quad g \in G.
\]
The composition of $\tau_{\cB}$ and the inclusion map of $\cB$ into $\ell^{\infty}(G)$ gives rise to an
left $G$-equivariant injective $*$-morphism and thus, by the Gelfand-Naimark correspondence,
to a left $G$-equivariant surjection 
\[
\pi_{\cB} : \beta G \ra \Delta(\cB),
\]
where $\beta G$ denotes the Stone-\v{C}ech compactification of $G$. Recall that the set
$\cM(G)$ of means on $G$ (or equivalently, of states on $\ell^{\infty}(G)$) corresponds
to the set $\cP(\beta G)$ of regular Borel probability measures on $\beta G$. Note that 
the map $\pi_{\cB}$ also induces a continuous affine map 
\[
(\pi_{\cB})_* : \cP(\beta G) \ra \cP(\Delta(\cB)).
\]
We denote the set of all positive elements in $\cB$ (i.e. the set of all non-negative functions on $G$
which belong to $\cB$) by $\cB_{+}$. Note that the transpose map
\[
\pi_{\cB}^* : C(\Delta(\cB)) \ra C(\beta G)
\]
is positive, i.e. $\pi_{\cB}^*\varphi$ is positive whenever $\varphi$ belongs to $\cB_{+}$.

\begin{proposition}
\label{reduction1}
Let $\cB \subset \ell^{\infty}(G)$ be a unital sub-C*-algebra and suppose $C \subset G$. For every 
$\lambda \in \cP(\beta G)$, there exists a $\lambda$-measurable set $C' \subset \Delta(\cB)$ with 
\[
\overline{C} \subset \pi_{\cB}^{-1}(C'),
\]
modulo $\lambda$-null sets, such that whenever $\varphi \in \cB_{+}$ satisfies
\[
\big\langle \pi_{\cB}^*\varphi , \chi_{\oC} \big\rangle_{L^2(\beta G,\lambda)} = 0,
\]
then
\[
\int_{C'} \varphi \, d(\pi_{\cB})_*\olambda = 0. 
\]
\end{proposition}

\begin{proof}
Fix $\lambda \in \cP(\beta G)$ and consider the indicator function $\chi_{\oC}$ as 
an element in $L^2(\beta G,\lambda)$. Note that  
\[
V := L^2(\Delta(\cB),(\pi_{\cB})_*\lambda)^{\pi_{\cB}} \subset L^2(\beta G,\olambda)
\]
is a closed subspace, consisting of those elements in $L^{2}(\beta G,\olambda)$ which are pull-backs of square-integrable elements on 
$\Delta(\cB)$ under the map $\pi_{\cB}^{*}$, so we can write 
\[
\chi_{\oC} = \rho_{\cB} + \rho_o
\]
where $\rho_{\cB}$ denotes the orthogonal projection of $\chi_{\oC}$ onto $V$. 
Since this projection is positive, the function $\rho_{\cB}$ is non-negative 
and 
\[
\big\langle (\pi_{\cB})^*\varphi, \chi_{\oC} \big\rangle_{\lambda} 
= 
\big\langle \varphi , \rho_{\cB} \big\rangle_{(\pi_{\cB})_*\olambda} = 0.
\]
Choose a $\lambda$-measurable representative of $\rho_{\cB}$ and define 
\[
C' := \Big\{ x \in \Delta(\cB) \, : \, \rho_{\cB}(x) > 0 \Big\} \subset \Delta(\cB).
\]
Clearly, the inclusion $\oC \subset \pi_{\cB}^{-1}(C')$ holds modulo $\lambda$-null sets
in $\beta G$ and since $\varphi$ is non-negative, we have
\[
\int_{C'} \varphi \, d(\pi_{\cB})_*\olambda = 0,
\]
which finishes the proof. 
\end{proof}
\subsubsection{Left amenable sub-C*-algebras}

We shall say that a left $G$-invariant sub-C*-algebra $\cB$ of $\ell^{\infty}(G)$ is 
\emph{left amenable} if there exists a left $G$-invariant state on $\cB$. Via the 
Gelfand-Naimark's correspondence, this is equivalent to saying that there exists
a $G$-invariant regular Borel probability measure on the Gelfand space $\Delta(\cB)$
of $\cB$. 

For an amenable sub-C*-algebra $\cB$ of $\ell^{\infty}(G)$, we let $\cL_{\cB}$ denote the 
(non-empty) weak*-compact and convex set of all states on $\ell^{\infty}(G)$ which restricts 
to a left invariant state on the algebra $\cB$. 

Let $B \subset G$ and let $\cB$ denote the \emph{Bebutov algebra} 
of $B$, i.e. the smallest left $G$-invariant unital sub-C*-algebra of $\ell^{\infty}(G)$ 
which contains the indicator function of $B$. In the introduction, we 
called a set $B$ \emph{strongly non-paradoxical} if the Bebutov algebra 
of $B$ is amenable and there exists $\lambda$ in $\cL_{\cB}$ which gives
positive measure to $B$. 

Since $\chi_B$ is idempotent in $\cB$, there exists a clopen subset $B' \subset \Delta(\cB)$ 
such that 
\[
\tau_{\cB}(\chi_{B'}) = \chi_{B}. 
\]
Recall from Subsection \ref{measure theory} that we also introduced the notation 
$\oB$ for the clopen set which $B$ corresponds to via the isomorphism between 
$\ell^{\infty}(G)$ and $C(\beta G)$. From the discussions above, it should be clear 
that 
\[
\oB = \pi_{\cB}^{-1}(B').
\]
Recall that the \emph{non-paradoxical density} of $B$ was defined earlier by
\[
d^*_{np}(B) := d^*_{\cL_\cB}(B).
\]
By Corollary \ref{max2}, there always exists a state $\lambda \in \cL_{\cB}$ which maximizes 
the right hand side. Let $\olambda$ denote the corresponding regular Borel probability measure
on $\beta G$ and set 
\[
\eta' := (\pi_{\cB})_*\olambda.
\]
Note that $\eta'$ is a $G$-invariant Borel probability measure on $\Delta(\cB)$. It is not hard to see
that $\eta'$ must be ergodic (since $\lambda$ is extremal) and
\[
d^*_{np}(B) = \eta'(B') = \sup_{\nu} \nu(B'),
\]
where the supremum is taken over all $G$-invariant probability measures on $\Delta(\cB)$.

\begin{corollary}
\label{paradoxical}
Let $A, B \subset G$ and suppose $B$ is strongly non-paradoxical. For every admissible 
$\mu \in \cP(G)$, there exist $\lambda \in \cP_\mu(\beta G)$ such that 
the measure $\eta' := (\pi_{\cB})_*\lambda$ on $\Delta(\cB)$ is $G$-invariant and ergodic with
\[
d^*_{np}(B) = \eta'(B') 
\]
and a Borel set $C' \subset \Delta(\cB)$ with
\[
\overline{AB} \supset \pi_{\cB}^{-1}(C')^c
\]
modulo $\lambda$-null sets in $\beta G$, such that
\[
\eta'(a \cdot B' \cap C') = 0, \quad \forall \, a \in A.
\]
\end{corollary}

\begin{proof}
From the discussion above, we know that there exists an ergodic and $G$-invariant
regular Borel probability measure $\eta'$ on $\Delta(\cB_B)$ with
\[
d^*_{np}(B) = \eta'(B').
\]
Fix $\mu \in \cP(G)$. Since $\eta'$ clearly belongs to $\cP_\mu(\Delta(\cB))$, Corollary 
\ref{relative} guarantees the existence of $\lambda \in \cP_\mu(\beta G)$ such that 
\[
(\pi_\cB)_*\lambda = \eta'.
\]
Since $\cB_B$ is left $G$-invariant, the closed subspace
\[
V := L^2(\Delta(\cB),(\pi_{\cB})_*\lambda)^{\pi_{\cB}} \subset L^2(\beta G,\lambda)
\]
is $G$-invariant as well, so the function $\pi_{\cB}^*\varphi_g$ belongs to $V$ for
all $g \in G$, where
\[
\varphi_g := \chi_{B'}(g^{-1} \cdot).
\]
Let $C := (AB)^c$ and note that
\[
\lambda(a \cdot \oB \cap \oC) =\big\langle \pi^*_{\cB} \, \varphi_a,  \chi_{\oC} \big\rangle_{L^2(\beta G,\lambda)} = 0,
\]
for all $a \in A$. By Proposition \ref{reduction1}, there exists a $\lambda$-measurable set
$C' \subset \Delta(\cB_B)$ with
\[
\oC \subset \pi_{\cB}^{-1}(C')
\]
modulo $\lambda$-null sets in $\beta G$ such that
\[
\int_{C'} \varphi_a \, d\eta' = \eta'(a \cdot B' \cap C') = 0
\]
for all $a \in A$.
\end{proof}

Recall from Section \ref{preliminaries} that the Fourier-Stiltjes algebra $\B(G)$ of a countable group 
$G$ is defined as the $*$-algebra generated by all matrix coefficients of unitary representations of $G$. 
The uniform closure of $\B(G)$ in $\ell^{\infty}(G)$ will here be referred to as the \emph{Eberlein algebra} of $G$ and we shall denote its (compact) Gelfand spectrum by $eG$. Let $\lambda_o$ denote the unique left and 
right invariant state on $\E(G)$ and write
\[
\E(G) = \AP(G) \oplus \E_o(G),
\]
where
\[
\E_o(G) := \Big\{ \varphi \in \E(G) \, : \, \lambda_o(|\varphi|^2) = 0 \Big\}.
\]
For a given $\varphi \in \E(G)$, we shall write 
\[
\varphi = \varphi_{ap} + \varphi_o,
\]
with $\varphi_{ap} \in \AP(G)$. 

Here, $\AP(G)$ denotes the C*-algebra of almost periodic functions on $G$. Its Gelfand spectrum $bG$ we shall refer to as the 
\emph{Bohr compactification} of $G$. As was mentioned in Section \ref{preliminaries}, this compact Hausdorff space can be endowed with a jointly continuous multiplication which gives $bG$ the structure of a 
compact Hausdorff group. In particuar, $bG$ admits a Haar probability measure $m$. 

Because of the decomposition above, we see that
\[
L^2(eG,\overline{\lambda}_o) \simeq L^2(bG,m)^{\pi},
\]
where $\pi : eG \ra bG$ is the $G$-equivariant map implementing 
the inclusion of $\AP(G)$ in $\E(G)$. Indeed, the function $\varphi_o$
above is  equivalent to the null function in $L^{2}(eG,\olambda_o)$. 

Furthermore, since $m$ is the unique $G$-invariant measure on $bG$, we have  
\[
\pi_*\overline{\lambda}_o = m.
\]
In particular, if $A \subset eG$ is a Borel set and $\varphi \in C(eG)$ is non-negative, 
with the property that
\[
\int_{A} \varphi \, d\overline{\lambda}_o = 0,
\]
then there exists a Borel set $A' \subset bG$ with $A \subset \pi^{-1}(A')$ such that
\[
\int_{A'} \varphi_{ap} \, dm = 0,
\]
In particular, there exists a Borel set $A'' \subset A'$ with 
$m(A'') = m(A')$ on which $\varphi_{ap}$ vanishes identically. \\

Recall that $\cFS$ denotes the weak*-compact and convex set of means on $G$ which restricts 
to $\lambda_o$ on $\B(G)$ (and thus on $\E(G)$). The following corollary is now an almost 
immediate consequence of the discussion above. 

\begin{corollary}
\label{fs cor}
Let $A \subset G$ be $\cFS$-large and suppose $\varphi \in \B(G)_{+}$ vanishes on $A$. Then
there exists a Borel set $A' \subset bG$ with
\[
m(A') \geq d^*_{\cFS}(A)
\]
such that $\varphi_{ap}$ vanishes on $A'$.
\end{corollary}

\begin{proof}
By Corollary \ref{max2}, there exists $\lambda \in \cFS$ such that
\[
d^*_{\cFS}(A) = \lambda'(A).
\]
Let $\bar{\varphi}$ denote the continuous function on $\beta G$ corresponding to $\varphi$. Note that $\bar{\varphi}$ is non-negative.
Since $\varphi \in \E(G)$, we  can write 
\[
\bar{\varphi} = \pi_{\E(G)}^* \, \varphi'
\]
for some $\varphi' \in C(eG)$, which is again non-negative. By assumption, 
\[
\big\langle \pi_{\E(G)}^*\varphi', \chi_{\oA} \rangle_{L^2(\beta G,\olambda)} = 0,
\]
and $\olambda_o = (\pi_{\E(G)})_*\olambda$. Hence, by Proposition \eqref{reduction1}, there exists a measurable set $A' \subset eG$ 
with  
\[
\olambda_o(A') \geq \overline{\lambda}(\oA) = d^*_{\cFS}(A)
\]
such that
\[
\int_{A'} \varphi' \, d\olambda_o = 0.
\]
The corollary now follows from the discussion above. 
\end{proof}

\subsection{Compact vectors and the Eberlein decomposition}
The aim of this subsection is to explicate the decomposition \ref{decomposition} for elements 
in $\B(G)$. We stress that the material covered here is quite standard (at least for 
abelian, or more generally amenable, groups), and is certainly well-known to experts.
We collect here a few basic results for completeness. \\
 
The basic setting can be described as follows. Let $(\cH,\rho)$ be a unitary representation on a Hilbert space $\cH$ (not necessarily separable). Let $\cH^{\rho}$ denote the closed subspace of $\rho$-fixed vectors and write $P_{\rho}$ for the corresponding orthogonal projection. 

Given $x,y \in \cH$, we define
the \emph{matrix coefficient at $(x,y)$} by
\[
\varphi^{\rho}_{x,y}(g) = \big\langle x, \rho(g) y \big\rangle, \quad g \in G.
\]
Recall that $\lambda_o$ denotes the unique left and right invariant mean on $\B(G)$. \\

The following extension of the (weak) von Neumann's Ergodic Theorem will be quite useful. 

\begin{lemma}
\label{mean ergodic}
For all $x, y \in \cH$, we have
\[
\lambda_o(\varphi^{\rho}_{x,y}) = \big\langle P_\rho x , P_{\rho} y \big\rangle.
\]
\end{lemma}

\begin{proof}
It suffices to establish the identity on a dense subspace of $\cH$. Note that the identity is 
trivial if either $x$ or $y$ belongs to $\cH^{\rho}$, and if $x$ (or $y$) is of the form
\[
x = z - \rho(g_o)^{-1}z
\] 
for some $z \in \cH$ and $g_o \in G$, then
\[
\varphi_{x,y}^{\rho}(g) = \langle z,\rho(g)y \rangle_{\cH} - \langle z,\rho(g_o \cdot g)y \rangle_{\cH} 
\]
for all $g \in G$. Since $\lambda_o$ is left-invariant, we see that in this case,
\[
\lambda_o(\varphi^{\rho}_{x,y}) = 0.
\]
Hence it suffices to show that if $x \in \cH$ and
\[
\langle x , z - \rho(g)^{-1} z \rangle_{\cH} = 0
\]
for all $z \in \cH$ and $g \in G$, then $x \in \cH^{\rho}$. Equivalently, 
for all $g \in G$ and $z \in \cH$, 
\[
\langle x - \rho(g)x , z \rangle_{\cH} = 0,
\]
which clearly implies that $\rho(g)x = x$ for all $g \in G$ and thus $x \in \cH^{\rho}$.
\end{proof}

\begin{remark}
Note that the proof does not utilize the uniqueness of $\lambda_o$. In fact,
this uniqueness is a consequence of the proof. 
\end{remark}

Recall that a vector $x \in \cH$ is \emph{compact} (with respect to $\rho$) if the orbit $\rho(G)x \subset \cH$ is pre-compact. Alternatively, the closure of the cyclic span of $x$ can be decomposed into a direct sum of finite-dimensional representations (by the Peter-Weyl Theorem). We denote the \emph{closed} subspace of $\rho$-compact vectors by $\cH_c$.

The following lemma provides a convenient way of producing compact vectors for a general representation. 

\begin{lemma}
\label{compact operator}
Suppose $z \in (\cH \otimes \cH)^{\rho \otimes \rho}$. Then the operator $K_z : \cH \ra \cH$ defined by
\[
\big\langle K_z x, y \big\rangle = \big\langle z, x \otimes y \big\rangle
\]
is compact and $\rho$-intertwining. In particular, $\textrm{Im}(K_z) \subset \cH_{c}$. 
\end{lemma}

\begin{proof}
Since $\cH \otimes \cH$ is the Hilbert space completion of the algebraic tensor product, there
exists a net $(z_\alpha)$ of finite sums of basic tensors such that $z_\alpha \ra z$ in the Hilbert 
tensor norm. In particular,
\begin{eqnarray*}
\sup_{\|x\|_{\cH} = 1} \|K_{z_\alpha}x - K_z x\|_{\cH} 
&=& 
\sup_{\|x\|_{\cH} = 1} \Big( \sup_{\|y\|_{\cH} = 1} 
\big| \langle z_{\alpha} - z, x \otimes y \rangle_{\cH \otimes \cH} \big| \Big) \\
&\leq&
\sup_{\|x\|_{\cH} = 1} \Big( \sup_{\|y\|_{\cH} = 1} 
\| z_\alpha - z \|_{\cH \otimes \cH} \cdot \|x \otimes y \|_{\cH \otimes \cH} \Big) \\
&=&
\| z_\alpha - z \|_{\cH \otimes \cH} \ra 0.
\end{eqnarray*}
Note that if
\[
z_\alpha = \sum_{i \in F_\alpha} x_{\alpha,i} \otimes y_{\alpha,i}
\]
for some finite set $F_\alpha$ and $x_{\alpha,i}, y_{\alpha,i} \in \cH$, then
\[
K_{z_\alpha} x = \sum_{i \in F_\alpha} \langle x_{\alpha_i} , x \rangle_{\cH} \cdot y_{\alpha_i},
\]
so in particular $K_{z_\alpha}$ has finite rank for all $\alpha$. Hence $K_{z_\alpha}$ is compact 
for every $\alpha$ and $K_z$
is the norm limit of a net of such operators, and thus compact. Finally, if 
$z \in (\cH \otimes \cH)^{\rho \otimes \rho}$, then
\begin{eqnarray*}
\langle K_z \rho(g)x , y \rangle_{\cH} 
&=& 
\langle z , \rho(g)x \otimes y \rangle_{\cH \otimes \cH} \\
&=&
\langle (\rho \otimes \rho)(g) z , \rho(g) x \otimes y \rangle_{\cH \otimes \cH} \\
&=&
\langle z , x \otimes \rho(g)^{-1} y \rangle_{\cH \otimes \cH} \\
&=&
\langle K_z x , \rho(g)^{-1}y \rangle_{\cH} \\
&=&
\langle \rho(g) K_z x , y \rangle_{\cH} 
\end{eqnarray*}
for all $g \in G$ and $x,y \in \cH$. We conclude that $K_z : \cH \ra \cH$ is $\rho$-intertwining,
which finishes the proof. 
\end{proof}

We can now prove an explicit form of the decomposition \ref{decomposition} for elements in 
$\B(G)$.

\begin{proposition}
\label{decomp}
Suppose $(\rho,\cH)$ is a unitary representation of $G$. If $v,w \in \cH$ and
\[
\varphi(g) = \big\langle \rho(g) v, w \big\rangle_{\cH} ,\quad g \in G,
\]
then $\varphi \in \B(G)$ and 
\[
\varphi_{ap}(g) = \big\langle \rho(g)P_cv, P_c w \big\rangle_{\cH}, \quad g \in G,
\]
where $P_c$ denotes the orthogonal projection onto $\cH_c$.
\end{proposition}

\begin{proof}
We wish to show that if $x, y \in \cH$, then the function defined by
\[
\phi_o(g) :=  \big\langle x, \rho(g) y \big\rangle_{\cH}
-
\langle P_c(x), \rho(g) P_c(y) \rangle_{\cH}
\]
belongs to $B_o(G)$. Since the function clearly belongs to $B(G)$, this is 
amounts to show that
\[
\lambda_o(|\phi_o|^2) = 0. 
\]
Define
\[
x_o := x - P_c x 
\qand
y_o := y - P_c y,
\]
and note that
\[
|\phi_o(g)|^2 
= 
\langle 
x_o \otimes x_o, 
(\rho \otimes \rho)(g) y_o \otimes y_o 
\rangle_{\cH \otimes \cH}.
\]
By Lemma \ref{mean ergodic}, we have
\[
\lambda_o(|\phi_o|^2) 
= 
\langle 
P_{\rho \otimes \rho} \, x_o \otimes x_o, 
P_{\rho \otimes \rho} \, y_o \otimes y_o
\rangle.
\]
Let $z := P_{\rho \otimes \rho}( x_o \otimes x_o)$ and note that 
we can now write 
\[
\lambda_o(|\phi_o|^2) = \langle K_z y_o, y_o \rangle,
\]
where the operator $K_z : \cH \ra \cH$ is defined as in Lemma \ref{compact operator}. 
By the same lemma, $K_z y_o \in \cH_c$, and since $y_o$ is orthogonal to $\cH_c$,
we conclude that
\[
\lambda_o(|\phi_o|^2) = \langle K_z y_o, y_o \rangle = 0,
\]
which finishes the proof.
\end{proof}

\subsection{Applications to product sets}

Let us now connect all this to our previous discussions about product sets. 

\subsubsection{Koopman representations}

Let $X$ be a compact $G$-space and suppose $\nu$ is a $G$-invariant and ergodic Borel probability
measure on $X$. The $G$-action on $X$ gives rise to a unitary representation $\rho$ 
(the so called \emph{Koopman representation}) of $G$ on $L^2(X,\nu)$ by 
\[
\rho(g) f = f(g^{-1} \cdot ), \quad g \in G.
\]
If $B, C \subset X$ are $\nu$-measurable sets, then the function
\[
\varphi(g) := \nu(g \cdot B \cap C) = \big\langle \rho(g)\chi_B, \chi_C \big\rangle_{L^2(X,\nu)}, \quad g \in G,
\]
is an element in $\B(G)$ (and hence in $\E(G)$) and by Proposition \label{decomp} we have
\[
\varphi_{ap}(g) = \big\langle \rho(g)P_c \chi_B, P_c \chi_C \big\rangle_{L^2(X,\nu)},
\] 
where $P_c$ denotes the projection onto the $\rho$-compact vectors in $L^2(X,\nu)$. 

\subsubsection{Kronecker-Mackey factors}

We shall need the following classical result of Mackey. 

\begin{proposition}[\cite{Ma}]
There exists a $\nu$-conull set $X' \subset X$, a compactification $(K,\iota_K)$ of $G$, a closed 
subgroup $K_o < K$ and a $G$-equivariant measurable map $\pi : X' \ra K/K_o$ such that 
\[
L^2(X,\mu)_{c} = L^2(K/K_o,m_o)^{\pi},
\]
where $m_o$ denotes the Haar probability measure on $K/K_o$. 
\end{proposition}

The compact homogeneous space $K/K_o$ is often referred to as the \emph{Kronecker} (or 
\emph{Kronecker-Mackey}) factor of $(X,\nu)$. \\

We shall use Mackey's result in the following way. Let $\psi_B$ and $\psi_C$ denote 
$m_o$-measurable pointwise realizations of $P_c \chi_B$ and $P_c \chi_C$ respectively,
and define
\[
\tilde{B} := \Big\{ k \in K/K_o \, : \psi_B(k) > 0 \Big\}
\qand
\tilde{C} := \Big\{ k \in K/K_o \, : \psi_C(k) > 0 \Big\}.
\]
Now assume that there is a set $A \subset G$ on which $\varphi_{ap}$ vanishes, i.e.
\[
\int_{K/K_o} \psi_B(\iota_K(a)^{-1} \cdot k) \, \psi_C(k) \, dm_o(k) = 0, \quad \forall \, a \in A.
\]
Since $\psi_B$ and $\psi_C$ are both non-negative functions, we must also have
\[
m_o(\iota_K(a) \cdot \tilde{B} \cap \tilde{C}) = 0, \quad \forall \, a \in A.
\]
Note that the function
\[
\psi(g) := m(\iota_K(g) \tilde{B} \cap \tilde{C}), \quad g \in G,
\]
is almost periodic. 

\subsubsection{Connections to $bG$}
By the universal property of the Bohr compactification $(bG,\iota_o)$, 
there exists a continuous surjective homomorphism 
\[
\iota' : bG \ra K
\]
such that $\iota_K = \iota' \circ \iota_o$. Let $p : K \ra K/K_o$ denote the canonical quotient map and set $p' := p \circ \iota'$. Then $p' : bG \ra K/K_o$ is a continuous and 
$G$-equivariant surjection, so in particular $p'_*m = m_o$, and the identity above can
be lifted to the identity
\[
m(\iota_o(a) \cdot p'^{-1}(\tilde{B}) \cap p'^{-1}(\tilde{C})) = 0, \quad \forall \, a \in A. 
\]
Now suppose that $A$ is a $\cFS$-large set. Since the function
\[
\varphi'(g) := m(\iota_o(g) \cdot p'^{-1}(\tilde{B}) \cap p'^{-1}(\tilde{C}))
\]
is almost periodic function on $G$, Corollary \ref{fs cor} implies that there exists a Borel
set $A_o \subset bG$ with 
\[
m(A_o) \geq d^*_{\cFS}(A)
\]
such that $\varphi'$ vanishes on $A_o$, so in particular, 
\[
\int_{bG} \int_{bG} \chi_{A_o}(a) \chi_{B_o}(a^{-1} \cdot k) \chi_{C_o}(k) \, dm(k) \, dm(a),
\]
where $B_o$ and $C_o$ denote the sets $p'^{-1}(\tilde{B})$ and $p'^{-1}(\tilde{C})$ respectively. Note that
\[
m(B_o) = m_o(\tilde{B}) \qand m(C_o) = m_o(\tilde{C}).
\]
Upon changing the order of integration, we get
\[
\int_{C_o} m(A_o \cap k \cdot B_o^{-1}) \, dm(k) = 0, 
\]
and thus $C_o \subset U^c$, modulo $m$-null sets in $bG$, where
\[
U := \Big\{ k \in bG \, : \, m(A_o \cap k \cdot B_o^{-1}) > 0 \Big\}. 
\]

\subsubsection{Almost everywhere inclusions}

Let us summarize the discussion so far. Given a $\cFS$-large set $A \subset G$ and 
$\nu$-measurable sets $B, C \subset X$, we have constructed a $m_o$-measurable
set $\tilde{C} \subset K/K_o$ such that
\[
C \subset \pi^{-1}(\tilde{C})
\] 
modulo $\nu$-null sets in $X$ and an open set $U \subset bG$ (of the explicit form above)
such that 
\[
p'^{-1}(\tilde{C}) \subset U^c
\]
modulo $m$-null sets in $bG$. 

If $Y$ is a $G$-space and $C$ is a subset of $Y$, we define
\[
C_y := \Big\{ g \in G \, : \, g \cdot y \in C \Big\}, \quad y \in Y. 
\] 
The following lemma will be useful. 

\begin{lemma}
\label{ae inclusion}
Let $Y$ be a compact $G$-space and suppose $\nu$ is a non-singular probability measure
on $Y$. If $A, B \subset Y$ are $\nu$-measurable sets with $\nu(A \setminus B) = 0$, then there 
exists a $\nu$-conull subset $Y' \subset Y$ such that
\[
A_y \subset B_y, \quad \forall \, y \in Y'.
\]
\end{lemma}

\begin{proof}
It suffices to prove that if $N \subset Y$ is a $\nu$-null set, then $N_y$ is empty for a conull set
of $y$ in $Y$. Note that
\[
\Big\{ y \in Y \, : \, N_y \neq \emptyset \Big\} = G \cdot N.
\]
Since $G$ is countable and $\nu$ is non-singular, we conclude that $\nu(G \cdot N) = 0$, which 
finishes the proof. 
\end{proof}

The Bohr compactification $(bG,\iota_o)$ can be viewed as a $G$-space in two ways (using the 
left or right action). In this paper, we shall use the left action, so if $g \in G$, we write
\[
g \cdot k := \iota_o(g) \cdot k, \quad k \in bG,
\]
so that if $V \subset bG$ is any subset, then 
\[
V_k = \Big\{ g \in G \, : \, g \cdot k \in V \Big\} = \iota_o^{-1}(V \cdot k^{-1}).
\]
The lemma above implies that there exists a $\nu$-conull subset $X'' \subset X'$ such that
\[
C_x \subset \pi^{-1}(\tilde{C})_x = \tilde{C}_{\pi(x)}, \quad \forall \, x \in X''
\]
and a $m$-conull set $Z \subset bG$ such that
\[
\tilde{C}_{p'(k)} = p'^{-1}(\tilde{C})_k \subset U_k^c = \iota_o^{-1}(U^c \cdot k^{-1}), \quad \forall \, k \in Z.
\]
In particular, for $\nu$-almost every $x \in X$, there exists $k \in bG$ such that
\[
C_x \subset \iota_o^{-1}(U^c \cdot k).
\]
We summarize this observation in the following proposition.
\begin{proposition}
\label{main help}
Suppose $A \subset G$ is $\cFS$-large and $B, C \subset X$ are $\nu$-measurable sets such
that the intersection $AB \cap C$ is a $\nu$-null set. Then there exist $m$-mesaurable sets 
$A_o, B_o \subset bG$ with
\[
m(A_o) \geq d^*_{\cFS}(A) \qand m(B_o) \geq \nu(B) 
\]
and a $\nu$-conull set $X'' \subset X$ with the property that for all $x \in X''$, there exists $k \in bG$
such that 
\[
C_x \subset \iota_o^{-1}(U^c \cdot k),
\]
where 
\[
U := \Big\{ k \in bG \, : \, m(A_o \cap k \cdot B_o^{-1}) > 0 \Big\}. 
\]
\end{proposition}

\subsection{Proof of Proposition \ref{bohr appears}}

Let us first briefly recall the setting of Proposition \ref{bohr appears}. Let $G$ be a countable 
group and fix an admissible probability measure $\mu$ on $G$ once and for all. Let 
$A \subset G$ be a Fourier-Stiltjes large set and let $B \subset G$ a strongly non-paradoxical
set. Denote by $\cB_B$ the Bebutov algebra of $B$ and by $X$ the Gelfand spectrum of 
$\cB_B$. Since $\cB$ is unital and separable, $X$ is compact and second countable. Let 
$B' \subset X$ denote the clopen set which corresponds to $B$ under the Gelfand map. \\

\subsubsection{Reducing the measure-preserving spaces}
By Corollary \ref{paradoxical} there exist $\lambda \in \cP_\mu(\beta G)$ such that the 
pushed forward probability measure $\eta' := (\pi_{\cB})_*\lambda$ on $X$ is $G$-invariant and ergodic
with
\[
d^*_{np}(B) = \eta'(B')
\]
and a $\lambda$-measurable set $C' \subset X$ with
\[
\overline{AB} \supset \pi_\cB^{-1}(C')^c
\]
modulo $\lambda$-null sets in $\beta G$, so in particular, we have 
\[
\eta'(a \cdot B' \cap C') = 0, \quad \, \forall \, a \in A.
\]
By Proposition \ref{existence}, $\lambda$ is non-singular and thus  Lemma \ref{ae inclusion} 
guarantees that we can find a $\lambda$-conull set $Z \subset \beta G$ such that 
\[
\overline{AB}_q \supset (\pi_\cB^{-1}(C')^c)_q = (C'^c)_{x},
\]
where $x = \pi_{\cB}(q)$, for all $q \in Z$.

\subsubsection{Reducing to $bG$}

By Proposition \ref{main help} applied with $\nu = \eta'$, there exist $m$-measurable 
sets $\tilde{A}, \tilde{B} \subset bG$ with
\[
m(\tilde{A}) \geq d^*_{\cFS}(A) \qand m(\tilde{B}) \geq \eta'(B')
\]
and a $\eta'$-conull subset $X' \subset X$ with the property that for all $x \in X'$, there is 
$k \in bG$ such that
\[
C'_x \subset \iota_o^{-1}(U^{c} \cdot k),
\]
where
\[
U := \Big\{ k \in bG \, : \, m(\tilde{A} \cap k \cdot \tilde{B}^{-1}) > 0 \Big\}. 
\]
Combining the two assertions above, we conclude that there exists a $\lambda$-conull set
$Z' \subset Z$ such that for all $q \in Z'$, there exists $k \in bG$ such that
\[
\overline{AB}_q \supset  C_x^{'c} \supset \iota_o^{-1}(U \cdot k),
\]
where $x = \pi_{\cB}(q)$, and this finishes the proof. 

\section{Proof of Proposition \ref{local inclusion}}
\label{Local}

We shall now outline the proof of Proposition \ref{local inclusion}. After some preliminary 
work, we will see that it easily follows from Corollary \ref{local inclusion0}.

\subsubsection{Almost automorphic points}

Let $Y$ be a compact $G$-space. A point $y \in Y$ is \emph{almost automorphic}
if whenever we have $g_\alpha \cdot y \ra y_o$ for some net $(g_\alpha)$ in $G$, then 
$g_\alpha^{-1} \cdot y_o \ra y$. 

We stress that this is a very special kind of point and such points do not always exist. However, 
it is easy to see that if $(K,\iota_o)$ is a compactification of $G$ and $Y := K/K_o$ for some 
closed subgroup $K_o < K$, where $G$ acts by
\[
g \cdot k K_o = g k K_o, \quad kK_o \in K/K_o,
\] 
then \emph{every} point in $Y$ is almost automorphic. 

\subsubsection{Constructing joint extensions}
Now let $X$ be another compact $G$-space, equipped with a $G$-transitive point $x_o$. 
If $Y$ admits an almost automorphic point, we shall see that it is easy to construct joint 
extensions of the two systems.

\begin{lemma}
\label{aa}
Suppose $y \in Y$ is an almost automorphic point. Then, for all $x \in X$,
there exists $y_o \in Y$ such that
\[
(x,y) \in \overline{G(x_o,y_o)}.
\]
\end{lemma}

\begin{proof}
Choose a net $(g_\alpha)$ in $G$ such that $g_\alpha \cdot x_o \ra x$ and extract a subnet
$(g_\beta)$ with the property that $g_\beta^{-1} \cdot y$ converges to a point $y_o \in Y$. 
Since $y$ is almost automorphic, we have $g_\beta \cdot y_o \ra y$, and thus
\[
g_\beta \cdot (x_o,y_o) \ra (x,y),
\]
along this subnet. 
\end{proof}

Fix $x \in X$ and an almost automorphic point $y \in Y$ and let 
\[
Y' := \overline{G \cdot y_o} \qand 
Z := \overline{G(x_o,y_o)},
\]
where $y_o$ is the point guaranteed by the previous lemma. In particular,
$z := (x,y) \in Z$. 

Let $z_o := (x_o,y_o)$ and denote by $\pi_X$ and $\pi_{Y'}$ the canonical projections 
of $Z$ onto $X$ and $Y'$ respectively. Let $\pi : \beta G \ra Z$ be the (left) $G$-equivariant 
map with $\pi(\bar{e}) = z_o$ induced from the inclusion $C(Z) \hookrightarrow C(\beta G)$.

\subsubsection{Small boundary property}

Fix an admissible probability measure $\mu$ on $G$ and suppose $U \subset Y'$  
is Jordan measurable with respect to all $\mu$-stationary probability measures on $Y'$. 
In particular, if $Y'$ is a compact homogeneous space as above, then this condition reduces 
to $U$ being Jordan measurable with respect to the Haar measure (since every 
$\mu$-stationary Borel probability measure on a compact homogenous space is invariant - see e.g. \cite{FG}).

Let $B \subset X$ be a clopen set and suppose
\[
B_x \supset U_y.
\]
Define
\[
B' := (\pi_{X} \circ \pi)^{-1}(B) \qand U' :=  (\pi_{Y'} \circ \pi)^{-1}(U)
\]
and choose $q \in \beta G$ such that $\pi(q) = z$. Then
\[
B'_q \supset U'_q.
\]
Note that $B' \subset \beta G$ is clopen and since $U'$ is the pull-back under a $G$-equivariant 
continuous map from $\beta G$ onto $Z$, Lemma \ref{lift} implies that $U'$ is Jordan measurable
with respect all elements in $\cL_\mu$.

In particular, by Corollary \ref{local inclusion0} (recall that $\cL_\mu$ is left saturated by Lemma \ref{left saturated}), there exists a $\cL_\mu$-conull set $T \subset G$ such that
\[
B'_{\bar{e}} \supset U'_{\bar{e}} \cap T.
\]
Since $T$ is right thick by Corollary \ref{mu thick}, we have proved 
Proposition \ref{local inclusion}.

\section{Proof of Proposition \ref{jordan}}
\label{Jordan}

Recall that if $K$ is a compact Hausdorff group and $U \subset K$ is an open set, then
there exists a finite set $F \subset K$ such that $FU = K$. We define the 
\emph{syndeticity index} $s_K(U)$ as the minimal cardinality of such a finite set. 

\subsection{Syndeticity index of product sets}

Let $K$ be a compact Hausdorff group. By Steinhaus Lemma, whenever $A, B \subset K$
are Borel sets with positive Haar measures, then the product set $AB$ has non-empty 
interior. We wish to bound the syndeticity index of this interior in terms of the Haar
measures of $A$ and $B$. 

We begin with a few simple lemmata.

\begin{lemma}
\label{help1}
Suppose $C, D \subset K$ are Borel sets. Then there exists $k_o \in K$
such that
\[
m(C \cap k_o \cdot D) \geq m(C) \cdot m(D).
\]
\end{lemma}

\begin{proof}
This is immediate from the identity 
\[
\int_K m(C \cap k \cdot D) \, dm(k) = m(C) \cdot m(D).
\]
\end{proof}

\begin{lemma}
\label{help2}
Suppose $E \subset K$ is a Borel set with positive Haar measure and define the open 
set
\[
U := \Big\{ k \in K \, : \, m(E \cap k \cdot E) > 0 \Big\} \subset K.
\]
Then,
\[
s_K(U) \leq \Big\lfloor \frac{1}{m(E)} \Big\rfloor.
\]
\end{lemma}

\begin{proof}
If $U = K$, there is nothing to prove, so assume there exists $k_1 \notin U$. 
For $n \geq 2$, we choose 
\[
k_n \notin U \cup \Big( \bigcup_{j = 1}^{n-1} k_j \cdot U \Big),
\]
provided the right hand side does not equal $K$. Note that $k_i^{-1} \cdot k_j \notin U$ for 
all $i < j$, so in particular,
\[
m(k_i \cdot E \cap k_j \cdot E) = m(E \cap k_i^{-1} \cdot k_j \cdot E) = 0,
\]
for all $i \neq j$. Hence
\[
m\Big( E \cup \Big( \bigcup_{j=1}^{n-1} k_j E \Big)\Big) = m(E) + \sum_{j=1}^{n-1} m(k_j \cdot E) = n \cdot m(E),
\]
for all $n$, provided that 
\[
K \neq U \cup\Big(  \bigcup_{j=1}^{n-1} k_j \cdot U \Big).
\]
Since $E \subset K$ has positive measure, there exists $n \geq 1$, such that
\[
n \cdot m(E) \leq 1 \qand (n+1) \cdot m(E) > 1.
\]
For this $n$, the set 
\[
F := \big\{e,k_1,\ldots,k_{n-1}\big\} \subset K
\]
must satisfy $F \cdot U = K$, and hence 
\[
s_K(U) \leq n \leq \frac{1}{m(E)},
\]
which finishes the proof.
\end{proof}

\begin{proposition}
\label{help3}
Suppose $C, D \subset K$ are Borel sets with positive Haar measures and define the
open set
\[
U := \Big\{ k \in K \, : \, m(C \cap k \cdot D) > 0 \Big\} \subset K.
\]
Then,
\[
s_K(U) \leq \Big\lfloor \frac{1}{m(C) \cdot m(D)} \Big\rfloor.
\]
\end{proposition}

\begin{proof}
By Lemma \ref{help1} there exists $k_o \in K$ such that the set
\[
E := C \cap k_o \cdot D
\]
has measure at least $m(C) \cdot m(D)$. Hence, 
\[
m(C \cap k \cdot D) \geq m(E \cap k \cdot k_o^{-1} \cdot E)
\]
for all $k \in K$. Define the open set
\[
U' := \big\{ k \in K \, : \, m(E \cap k \cdot E) > 0 \big\} \subset K.
\] 
By Lemma \ref{help2} there exists a finite set $F \subset K$ with
\[
|F| \leq \Big\lfloor \frac{1}{m(E)} \Big\rfloor \leq \Big\lfloor \frac{1}{m(C) \cdot m(D)} \Big\rfloor.
\]
such that $F U' = K$. Since 
\[
U \supset U' \cdot k_o,
\]
we conclude that 
\[
s_K(U) \leq \Big\lfloor \frac{1}{m(C) \cdot m(D)} \Big\rfloor.
\]
\end{proof}

\subsection{Proof of Proposition \ref{jordan}}

There is no reason to expect that the set $U$ in Proposition \ref{help3} is Jordan measurable. 
However, as we shall see in this subsection, it always contains an open Jordan measurable subset 
of the same syndeticity index. This will follow from the following proposition, which for compact
and \emph{metrizable} groups is quite standard. Although the setting can be adapted so that only this
weaker version of the proposition needs to be applied, we shall here give the proof of the same
statement for \emph{any} compact Hausdorff group.  

\begin{proposition}
\label{pre-jordan}
Suppose $C \subset K$ is a closed set. Then $C$ equals the intersection of a decreasing 
net $(C_\alpha)$ of closed Jordan measurable sets in $K$. 
\end{proposition}

Before we turn to the proof of this proposition, we give an easy corollary. 

\begin{corollary}
\label{cor. pre-jordan}
Suppose $U \subset K$ is an open set with the property that there exists a finite set $F \subset K$ 
such that $FU = K$. Then there exists an open Jordan measurable subset $U' \subset U$ such 
that $FU' = K$.
\end{corollary}

\begin{proof}
By Proposition \ref{pre-jordan}, there exists an increasing net $(V_\alpha)$ of open Jordan 
measurable sets with
\[
U = \bigcup_\alpha V_\alpha. 
\]
In particular,
\[
K = FU = \bigcup_\alpha FV_\alpha.
\]
Hence, $(FV_\alpha)$ is an open covering of $K$, so we may extract a finite sub-covering. 
Since $(V_\alpha)$ is nested, there must exist $\alpha$ such that
\[
K = FV_\alpha,
\]
which finishes the proof. 
\end{proof}

The proof of Proposition \ref{jordan} is now immediate. 

\begin{proof}[Proof of Proposition \ref{jordan}]
Suppose $A, B \subset K$ are Borel sets with positive $m$-measures and define the open set
\[
U := \Big\{ k \in K \, : \, m(A \cap k \cdot B) > 0 \Big\} \subset K.
\]
By Lemma \ref{help3}, 
\[
s_{K}(U) \leq \Big\lfloor \frac{1}{m(A) \cdot m(B)} \Big\rfloor,
\]
and by Corollary \ref{cor. pre-jordan}, we can find a Jordan measurable set $U' \subset U$
with
\[
s_{K}(U') \leq \Big\lfloor \frac{1}{m(A) \cdot m(B)} \Big\rfloor,
\]
which finishes the proof.
\end{proof}

\subsection{Proof of Proposition \ref{pre-jordan}}

Recall (see Chapter 2 in \cite{Ru} for the proof) that if $X$ is any (locally) compact Hausdorff space and $K \subset X$ is compact and $U \subset X$ is an open set containing $K$, then there exists an open subset $V$ with compact closure such that
\[
K \subset V \subset U.
\]
This simple property of locally compact Hausdorff spaces is the only thing required to give a 
proof of the following statement, which clearly implies Proposition \ref{pre-jordan}.

\begin{proposition}
\label{real jordan}
Suppose $K \subset X$ is compact. Then, for every open subset $U \supset K$, there exists a
$\mu$-Jordan measurable set $K \subset C \subset U$.
\end{proposition}

\begin{proof}
Fix $\eps > 0$ and a decreasing sequence $(\eps_n)$ of positive real numbers, converging 
to zero, with $\eps_o = \eps$. 

By regularity of $\mu$, we can find an open set $U_o \subset U$ with 
\[
K \subset U_o
\]
and 
\[
\mu(U_o) \leq \mu(K) + \eps_0. 
\]
Since $X$ is a locally compact Hausdorff space, there exists an open set $V_{0,1} \subset X$ with 
compact closure such that
\[
K \subset V_{0,1} \subset \overline{V_{0,1}} \subset U_o. 
\]
By regularity of $\mu$, we choose an open set $U_1 \subset X$ such that
\[
\mu(U_1) \leq \mu(\overline{V_{0,1}}) + \eps_1
\]
and since $X$ is a locally compact Hausdorff space we can find open sets $V_{1,0}, V_{1,1} \subset X$
such that
\[K \subset \overline{V_{0,1}} 
\subset V_{1,0} \subset \overline{V_{1,0}} 
\subset V_{1,1} \subset \overline{V_{1,1}}
\subset U_1.
\]
Again, by regularity of $\mu$, we can find an open set $U_2 \subset V_{1,1}$ with
\[
\mu(U_2) \leq \mu(\overline{V_{1,0}}) + \eps_2
\]
and open sets $V_{2,0}, V_{2,1} \subset X$ such that
\[
K \subset \overline{V_{1,0}} 
\subset V_{2,0} \subset \overline{V_{2,0}} 
\subset V_{2,1} \subset \overline{V_{2,1}}
\subset U_2 \subset V_{1,1}.
\]
For $n > 2$, we can find an open set $U_{n+1} \subset V_{n,1}$ with
\[
\mu(U_{n+1}) \leq \mu(\overline{V_{n,0}}) + \eps_{n+1}
\]
and open sets $V_{n+1,0}, V_{n+1,1} \subset X$ such that
\[
K \subset \overline{V_{n,0}} 
\subset V_{n+1,0} \subset \overline{V_{n+1,0}} 
\subset V_{n+1,1} \subset \overline{V_{n+1,1}}
\subset U_{n+1} \subset V_{n,1}.
\]
By construction
\[
V_{n,0} \subset V_{n+1,0} \qand V_{n,1} \supset V_{n+1,1},
\]
for all $n$ and hence, if we define
\[
V := \bigcup_{n} V_{n,0} \qand C := \bigcap_{n} \overline{V_{n,1}},
\]
then $V \subset X$ is open and $C \subset X$ is closed, and moreover
\[
\mu(V) = \lim_n \mu(V_{n+1,0}) \qand \mu(C) = \lim_{n} \mu(\overline{V_{n+1,1}}).
\]
Since 
\[
\mu(\overline{V_{n+1,1}}) \leq \mu(U_{n+1}) \leq \mu(\overline{V_{n,0}}) + \eps_{n+1} \leq \mu(V_{n+1,0}) + \eps_{n+1},
\]
for all $n$ and $\eps_n \ra 0$, we have
\[
\mu(C) = \lim_{n} \mu(\overline{V_{n+1,1}}) \leq \lim_{n} \big( \mu(V_{n+1,0}) + \eps_{n+1} \big)
= \mu(V). 
\]
Since $V \subset C$, we conclude that $C$ is $\mu$-Jordan measurable, and
\[
\mu(C) \leq \mu(U_o) \leq \mu(K) + \eps.
\]
To prove the second assertion, we first note that the set $K \subset V_o$ is arbitrary, and thus
we can always ensure that there exists a closed $\mu$-Jordan measurable set 
\[
K \subset C \subset V_o.
\]
Since $K$ is compact, there exists a decreasing net $(V_\alpha)$ of open sets with 
\[
K = \bigcap V_\alpha,
\]
and from the arguments above, we know that, for every $\alpha$, there exists a $\mu$-Jordan
measurable set such that
\[
K \subset C_\alpha \subset V_\alpha,
\]
which finishes the proof. 
\end{proof}

\end{document}